\documentclass[a4paper,10pt]{amsart}
\usepackage{amsmath, amsthm, amssymb}
\usepackage{lipsum}
\usepackage{tikz}
\usepackage{physics}
\usepackage{bm}
\usepackage{graphicx}
\usepackage{todonotes}
\usepackage[hyperindex=true,frenchlinks=true,colorlinks=true,linktocpage]{hyperref}
\usepackage{enumitem}
\usepackage[super]{nth}

\newtheorem{thm}{Theorem}[section]
\newtheorem{lem}[thm]{Lemma}
\newtheorem{definition}[thm]{Definition}

\newtheorem{proposition}[thm]{Proposition}
\newtheorem{remark}[thm]{Remark}
\newtheorem{algo}[thm]{Algorithm}
\newtheorem{corollary}[thm]{Corollary}

\def\SG{{\mathfrak S}}  %
\def\Ig{{\bf I}}        %
\def\Jg{{\bf J}}        %
\def\Hg{{\bf H}}        %
\def\ie{{\emph i.e.}}   %
\def\ssig{{\bm \sigma}} %
\def\ttau{{\bm \tau}}   %
\def\etiq{\operatorname{label}}
\def\tw{\operatorname{tw}}
\def\invol{\iota}
\def\LagInvol{\operatorname{i}}

\def\Prob{\operatorname{Prob}} %
\def\Des{\operatorname{Des}}   %
\def\Seg{\operatorname{Seg}}   %
\def\Sign{\operatorname{Sign}} %
\def\GDes{\operatorname{GDes}} %
\def\GC{\operatorname{GC}}     %
\def\ade{{\rm ADE}}            %
\def\312{31\!-\!2}
\def\graybullet{{\color{gray!50}\bullet}}

\title[Combinatorics of the 2-species exclusion processes]{Combinatorics of the 2-species exclusion processes,
  marked Laguerre histories, and partially signed permutations}

\author[Corteel]{Sylvie Corteel}
\address{Department of Mathematics, UC Berkeley, USA}
\email{corteel@berkeley.edu}
\author[Nunge]{Arthur Nunge}
\address{Department of Mathematics, Wien Universit\"at, Austria}
\email{arthur.nunge@univie.ac.at}


\usepackage[backend=bibtex,maxcitenames=50,maxnames=50]{biblatex}
\begin{document}

\maketitle

\begin{abstract}
   Starting from the two-species partially asymmetric simple exclusion
process, we study
   a subclass of signed permutations, the partially signed permutations,
   using the combinatorics of Laguerre histories. From this physical
   and bijective point of view, we obtain a natural  descent statistic on
partially signed permutations;
   as well as partially signed permutations patterns.
\end{abstract}

\section{Introduction}

The two-species partially asymmetric simple exclusion process
(2-PASEP) is a Markov chain with two types of particles ($\bullet$ and
$\graybullet$) and holes ($\circ$).  Particles can hop to the right
and to the left and particles of type $\bullet$ can enter and exit the
system.  If there are no particles of type $\graybullet$, we recover
the classical PASEP. See Section~\ref{secPASEP} and~\cite{uchiyama}
for a detailed definition.  The classical PASEP has given rise to
beautiful combinatorics related to Laguerre histories \cite{JV},
permutations \cite{Corteel,JV,SW}, permutation tableaux
\cite{CW0,CW1,SW}, alternative tableaux \cite{Nadeau} and staircase
tableaux \cite{CW2} in its most general case.  All these objects are
shown to be connected to the PASEP thanks to the Matrix Ansatz \cite{DEHP} and the fact
that the partition function of the model is related to the moments of
the Askey Wilson polynomials \cite{USW}.
In special cases, we can define a Markov chain on the permutations
or the tableaux that projects to the ASEP chain. See \cite{CW1}
for example. This is classical in Markov chain theory and is called lumpability
\cite{bookmarkov,lump}.

In the case of the 2-PASEP the Matrix Ansatz extends naturally
\cite{uchiyama}. We shall detail this in Section~\ref{secPASEP}.  The
partition function is now related to the mixed moments of the Askey
Wilson polynomials \cite{CW}.  It is therefore expected (but not at
all trivial) to see that the combinatorics extends to this generalized
model.  The alternative tableaux become rhombic alternative tableaux
(RATS) \cite{MV} and the permutations become assembl\'ees of
permutations \cite{MV1}. In the most general case we get rhombic
staircase tableaux \cite{CMW}. 

In this paper we take a slightly different approach, as some of the
statistics coming from the 2-PASEP are not natural on the assembl\'ees
of permutations. We generalize the results of \cite{Corteel,JV}
related to Laguerre histories and permutations.  In these cases the
states of the PASEP are in bijection with compositions and the
statistic coming from the PASEP is in bijection with the weight of the
paths or equivalently the number of the generalized patterns $(31-2)$
of the permutation.  Generalized patterns were first introduced
in a general framework by Babson and Steingrímsson in \cite{[1]}, but
some instances had been treated previously in various contexts. For
example, the pattern $(31-2)$ is implicit in \cite{[10]} and in
\cite{[17]}.
We define some generalization of Laguerre histories~:
the {\em marked Laguerre histories}.  We give a bijection between the
{marked Laguerre histories} and a subclass of signed permutations
where we do not put a sign on 1.  We call them {\em partially signed
permutations}.  For example, the partially signed permutations of size
2 are $12$, $21$, $1\overline{2}$ and $\overline{2}1$.  The states of
the 2-PASEP are in bijection with {\em segmented compositions} and the
statistic coming from the 2-PASEP is in bijection with the weight of
the paths or equivalently the number of some generalized patterns of
the partially signed permutation.

All the detailed definitions are given in Section~\ref{defi}. In
Section~\ref{secPASEP}, we give a new solution for the interpretation
of the probabilities of the 2-PASEP. In Section~\ref{sec:paths}, we
use this solution to interpret these probabilities using marked
Laguerre histories. In Section~\ref{sec:perm}, we use the
Françon-Viennot bijection to obtain an interpretation in terms of
generalized permutations. In Section~\ref{sectionLargeLaguerre}, we
give another interpretation in terms of large Laguerre histories and
give an explicit involution on these large Laguerre histories that
explains the so-called particle hole symmetry of the process.  We end
the paper with some final comments and open problems in
Section~\ref{conclusion}. 

{\bf Acknowledgements.} The authors would like to thank Matthieu Josuat-Verg\`es and Lenny Tevlin for
helpful discussions during the elaboration of this work. They would also thank the anynymous referee
for his contructive comments. Finally they thank IRIF, CNRS and Universit\'e de Paris where this work
was elaborated.

\section{Notations and definitions}
\label{defi}

\subsection{Signed permutations and segmented compositions}

A \emph{signed permutation} $\ssig$ of size $n$ is a permutation of
$n$ such that each value has a sign plus or minus. We denote by $B_n$
the set of signed permutations of size $n$.  We overline negative
values and we say that $\overline{k} \in \ssig$ if the value $k$ has a
negative sign in $\ssig$. For example, $\ssig =
\overline{2}57836\overline{4}1$ is a signed permutation of size $8$.

When we compare two values $\ssig_i$ and $\ssig_j$ of a signed
permutation $\ssig$, we use the order $\overline{1} < 1 < \overline{2}
< 2 < \cdots$.

\begin{definition}
  A \emph{partially signed permutation} is a signed permutation
  where~$1$ is not signed. We denote by~$B_n'$ the set of these
  permutations.

  The \emph{overlined values} of a partially signed permutation are
  its negative values. We denote by~$\Sign(\ssig)$ the set of all
  overlined values of a partially signed permutation~$\ssig$.
\end{definition}
For example, $\ssig=\overline{2}57836\overline{4}1$ is an element of
$B_n'$ and its set of overlined values is
$\Sign(\overline{2}57836\overline{4}1)=\{2,4\}$.

We define generalized patterns~\cite{[1]} for partially signed
permutations.
\begin{definition}
  \label{def:312}
  A $\312$ pattern of a partially signed permutation $\ssig$ is
  a pair $(\sigma_i\sigma_{i+1},\sigma_j)$ such that $j > i+1$ and
  $\ssig_{i} > \ssig_j > \ssig_{i+1}$. We denote this pattern by
  $\ssig_i\ssig_{i+1}\!-\!\ssig_j$.
  
  A $(31,\overline{2})$ pattern is a pair
  $(\sigma_i\sigma_{i+1},\overline{k})$ such that $\overline{k} \in
  \ssig$ and $\ssig_{i} \geq \overline{k}> \ssig_{i+1}$.

  We denote by $\tw(\ssig)$ the number of $\312$ patterns of $\ssig$
  plus its number of $(31,\overline{2})$ patterns
\end{definition}
Note that in the second case the value $\overline{k}$ can be to the
left or to the right of $\sigma_i$. For example, the $\312$ patterns
of $\ssig = \overline{2}57836\overline{4}1$ are $83\!-\!6$ and
$83\!-\!\overline{4}$ and the $(31,\overline{2})$ patterns of $\ssig$
are $(83, \overline{4})$, $(\overline{4}1,\overline{2})$, and
$(\overline{4}1, \overline{4})$.

\begin{definition}
  A \emph{segmented composition} of an integer $n$ is a finite
  sequence~$\Ig$ of~$\ell$ positive integers $(i_1,\ldots ,i_\ell)$
  that sum to~$n$ separated by vertical bars or commas.
  
  The \emph{descent set} of $\Ig$ (denoted by $\Des(\Ig)$) is the set of
  values $i_1+i_2+\dots+i_{k}$ where $i_k$ is
  followed by a comma in $\Ig$. Similarly, the \emph{segmentation set}
  of $\Ig$ (denoted by $\Seg(\Ig)$) is the set of values
  $i_1+i_2+\dots+i_k$ such that $i_k$ is followed by a bar in~$\Ig$.
\end{definition}
For example, 
\begin{equation}
  (\Des(1|2|1,2,2),\Seg(1|2|1,2,2))=(\{4, 6\},\{1,3\}).
\end{equation}
The \emph{ADE-word} associated with $\Ig$ is the word $w$ of
size $n-1$ such that
\begin{equation}
  \begin{array}{rcl}
    i \in \Des(\Ig) & \Rightarrow & w_i = E; \\
    i \in \Seg(\Ig) & \Rightarrow & w_i = A; \\
    i \notin \Des(\Ig), i \notin \Seg(\Ig) & \Rightarrow & w_i = D.
  \end{array}
  \label{ade}
\end{equation}
We denote this word by $\ade(\Ig)$. For example,
$\ade(1|2|1,2,2) = ADAEDED$.

\begin{definition}
  The \emph{Genocchi descent set} of a partially signed
  permutation $\ssig$ of size $n$ is the set of positive values
  followed by a smaller value. In other words,
  \begin{equation}
    \GDes(\ssig) := \{i~|~i\notin\Sign(\ssig),~\ssig_j = i \Rightarrow
    \ssig_j > \ssig_{j+1}\} 
  \end{equation}
  The \emph{Genocchi composition of descents} of a partially
  signed permutation (denoted by $\GC(\ssig)$) is the
  segmented composition $\Ig$ whose descent set is
  $\{d-1~|~d\in\GDes(\ssig)\}$ and whose segmentation set is
  $\{s-1~|~s\in\Sign(\ssig)\}$.
\end{definition}
Note that if $\ssig$ does not have any overlined values, the statistic
$\GC$ is the composition of the values of descents minus one and is
the same as the one defined in~\cite{HNTT}. For example, if $\ssig =
\overline{2}57836\overline{4}1$, we have $\GDes(\ssig)=\{6,8\}$,
$\Sign(\ssig)=\{2,4\}$ so we have $\GC(\ssig)=(1|2|2,2,1)$.

\subsection{Laguerre histories}

Recall that a Motzkin path of size $n$ is a
path going from $(0,0)$ to $(n,0)$ using increasing steps, decreasing
steps, and horizontal steps that never goes below the horizontal
axis. For any path $P$, we denote by $P_i$ the $i$'th step of $P$. We
call the starting (resp. ending) height of a step, the distance
between the beginning (resp. end) of this step and the horizontal
axis. Unless otherwise specified, the height is understood to be the
starting height.

A Laguerre history is a weighted object introduced by Viennot
in~\cite{Vi}, see also~\cite{dMV}. The Laguerre histories are in
bijection with permutations through the Fran\c con-Viennot
bijection~\cite{FV}. These objects have been used to study some
properties of the ASEP~\cite{JV}.

\begin{definition}
  A \emph{Laguerre history} $H$ of size $n$ is a weighted Motzkin path
  of size $n$ with two different horizontal steps such that
  \begin{itemize}
  \item a $\nearrow$ or $\longrightarrow$ step starting from height
    $h$ has a weight between 0 and $h$;
  \item a $\searrow$ or $\dashrightarrow$ step starting from height
    $h$ has a weight between 0 and $h-1$.
  \end{itemize}
  We denote by ${\rm tw}(H)$ the total sum of the weights of $H$.

  If $w_i$ is the weight of the $i\!$'th step of a Laguerre history $L$
  of size $n$, we call $w=w_1\dots w_n$ the \emph{weight} of $L$.
\end{definition}

An example of a Laguerre history of size $8$ is given in
Figure~\ref{fig:lagHist}. To avoid cumbersome figures, we only
represent the non-zero weights of the steps on the figures.

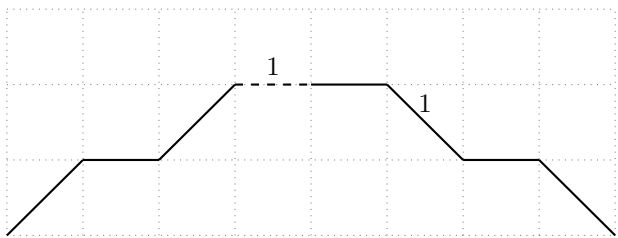
\begin{figure}[ht]
  \begin{center}
    \begin{tikzpicture}
  \draw [thin, gray, dotted] (0, 0) grid (8, 3);
  \draw [thick] (0, 0) -- (1, 1) node[midway,scale=1,above]{};
  \draw [thick] (1, 1) -- (2, 1) node[midway,scale=1,above]{};
  \draw [thick] (2, 1) -- (3, 2) node[midway,scale=1,above]{};
  \draw [thick, dashed] (3, 2) -- (4, 2)node[midway,scale=1,above]{$1$};
  \draw [thick] (4, 2) -- (5, 2) node[midway,scale=1,above]{};
  \draw [thick] (5, 2) -- (6, 1) node[midway,scale=1,above]{$1$};
  \draw [thick] (6, 1) -- (7, 1) node[midway,scale=1,above]{};
  \draw [thick] (7, 1) -- (8, 0) node[midway,scale=1,above]{};
\end{tikzpicture}
    \caption{An example of Laguerre history of size $8$ and weight
      $00010100$.}
    \label{fig:lagHist}
  \end{center}
\end{figure}

The Fran\c con-Viennot bijection~\cite{FV} is a bijection between
Laguerre histories of size $n$ and permutations $\sigma$ in $\SG_n$.
We denote  this map by $\psi_{FV}$.
We compare each value of the permutation $\sigma$ with its two neighbors.
We use the convention $\sigma_0 = 0$ and $\sigma_{n+1} = n+1$.

\begin{algo}\cite{FV}
  \label{algoFV}
  \begin{itemize}[before = \leavevmode, itemsep=0pt]
  \item Input: A permutation~$\sigma\in \SG_n$.
  \item Output: A Laguerre history $H$ of size~$n$.
  \item Execution: For $k$ from 1 to $n$, let $j$ be such
    that $\sigma_j=k$. The $k\!$'th step of~$H$~is
    \begin{itemize}
    \item $H_{k} =~ \nearrow$  if $\sigma_j$ is a \emph{valley},
      \ie, $\ssig_{j-1} > \sigma_j < \sigma_{j+1}$,
    \item $H_{k} =~ \searrow$ if $\sigma_j$ is a \emph{peak},
      \ie, $\ssig_{j-1} < \sigma_j > \sigma_{j+1}$,
    \item $H_{k} =~ \longrightarrow$  if $\sigma_j$ is a \emph{double rise},
      \ie, $\ssig_{j-1} < \sigma_j < \sigma_{j+1}$,
    \item $H_{k} =~ \dashrightarrow$ if $\sigma_j$ is a \emph{double descent},
      \ie, $\ssig_{j-1} > \sigma_j > \sigma_{j+1}$.
    \end{itemize}
    The weight of $k\!$'th step of $H$ is equal to the number of $\312$
    patterns such that $k$ is the number corresponding to~$2$
    in~$\sigma$.
  \end{itemize}
\end{algo}

The Laguerre history in Figure~\ref{fig:lagHist} is the image of the
permutation~$\sigma=25783641$. Indeed, the valleys of $\sigma$ are $3$
and $1$; its peaks are $8$ and $6$; its double rises are~$2$,~$5$,
and~$7$; and its only double descent is~$4$. Finally, its $\312$
patterns are $83\!-\!4$ and~$83\!-\!6$.

We shall also need the reciprocal map of the Fran\c con-Viennot
bijection described by the following algorithm.

\begin{algo}
  \label{algoInverseFV}
  \begin{itemize}[before = \leavevmode, itemsep=0pt]
  \item Input: A Laguerre history $H$ of size~$n$
  \item Output: A permutation $\sigma$ of size~$n$.
  \item Initialization: $\sigma = \circ$;
  \item Execution:
    Let $w$ be the weight of $H$. For $k\in\{1,\dots, n\}$, replace the
    $(w_k+1)$-st $\circ$ of $\sigma$ by:
    \begin{itemize}
    \item $\circ k \circ$ if $H_k = \nearrow$;
    \item $k \circ$ if $H_k = \longrightarrow$;
    \item $\circ k$ if $H_k = \dashrightarrow$;
    \item $ k$ if $H_k = \searrow$.
    \end{itemize}
    The final permutation is obtained by removing the last $\circ$.
  \end{itemize}
\end{algo}

For example, see~\eqref{exFVrecip} for a step by step
execution of Algorithm~\ref{algoInverseFV} on the Laguerre history of
Figure~\ref{fig:lagHist}.

\begin{equation}
  \label{exFVrecip}
  \begin{array}{c}
    \sigma = \circ ~\to~ \circ\, 1\, \circ ~\to~
    2 \circ 1\, \circ ~\to~
    2 \circ 3 \circ 1 ~\to~
    2 \circ 3 \circ 41\, \circ\vspace{0.2cm}\\ \to~
    25 \circ 3 \circ 41\, \circ ~\to~
    25 \circ 3641\, \circ ~\to~
    257 \circ 3641\, \circ 
    \vspace{0.2cm}\\ \to~25783641\, \circ ~\to~
    25783641
  \end{array}
\end{equation}

\section{The 2-PASEP and the Matrix Ansatz}
\label{secPASEP}

The two-species PASEP is a Markov chain whose states are words of
length $N$ in the letters $\{\circ, \bullet, \graybullet\}$. This was
first studied in a 
more general setting in~\cite{uchiyama} and then combinatorialy
in~\cite{CMW,CW,MV,MV1}. This Markov chain is described the following
way:

\begin{definition}
Let $q$ be a constant such that $0 \leq q \leq 1$. The 2-PASEP is the
Markov chain on the words in the letters $\circ, \bullet, \graybullet$
with transition probabilities:
\begin{itemize}
\item  If $x = A\bullet \circ B$ and
$y = A \circ \bullet B$ then
$P_{x,y} = \frac{1}{N+1}$ (black particle hops right) and
$P_{y,x} = \frac{q}{N+1}$ (black particle hops left).
\item  If $x = A\graybullet \circ B$ and
$y = A \circ \graybullet B$ then
$P_{x,y} = \frac{1}{N+1}$ (gray particle hops right) and
$P_{y,x} = \frac{q}{N+1}$ (gray particle hops left).
\item  If $x = A\bullet\graybullet B$ and
$y = A \graybullet \bullet B$ then
$P_{x,y} = \frac{1}{N+1}$ (black particle hops right) and
$P_{y,x} = \frac{q}{N+1}$ (black particle hops left).
\item  If $x = \circ B$ and $y = \bullet B$
then $P_{x,y} = \frac{1}{N+1}$ (particle enters from left).
\item  If $x = B \bullet$ and $y = B \circ$
then $P_{x,y} = \frac{1}{N+1}$ (particle exits to the right).
\item  Otherwise $P_{x,y} = 0$ for $y \neq x$
and $P_{x,x} = 1 - \sum_{x \neq y} P_{x,y}$.
\end{itemize}
\end{definition}

An example of a chain on three letters among which two are $\graybullet$
is given on Figure~\ref{chain} where we represent the transitions
$(N+1)P_{X,Y}$ for $X \ne Y$ and $P_{X,Y}\ne 0$.

\begin{figure}[ht]
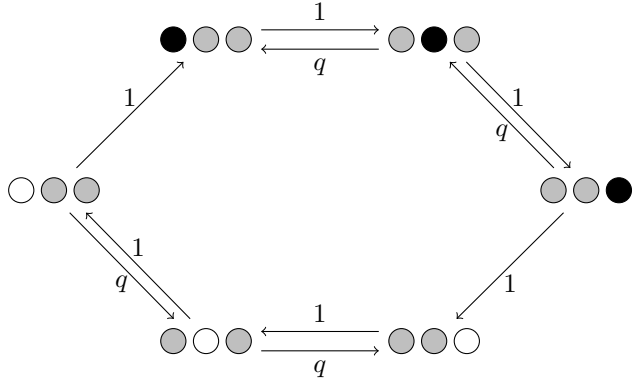

  \begin{center}
    \begin{tikzpicture}
  \node (EAA) at (0,0) {\input{figures/stateEAA.tex}};
  \node (DAA) at (2,2){\input{figures/stateDAA.tex}};
  \node (ADA) at (5,2){\input{figures/stateADA.tex}}; 
  \node (AAD) at (7,0){\input{figures/stateAAD.tex}};
  \node (AEA) at (2,-2){\input{figures/stateAEA.tex}}; 
  \node (AAE) at (5,-2){\input{figures/stateAAE.tex}};

  \draw[->] (EAA) -- (DAA) node[midway,above] {$1$};
  \draw[->] (AAD) -- (AAE) node[midway,below] {$1$};

  \draw[->] (DAA.10) -- (ADA.170) node[midway,above] {$1$};
  \draw[->] (ADA.-170) -- (DAA.-10) node[midway,below] {$q$};

  \draw[->] (AAE.170) -- (AEA.10) node[midway,above] {$1$};
  \draw[->] (AEA.-10) -- (AAE.-170) node[midway,below] {$q$};
  
  \draw[->] (ADA.-35) -- (AAD.125) node[midway,above] {$1$};
  \draw[->] (AAD.145) -- (ADA.-55) node[midway,below] {$q$};
  
  \draw[->] (AEA.125) -- (EAA.-35) node[midway,above] {$1$};
  \draw[->] (EAA.-55) -- (AEA.145) node[midway,below] {$q$};

\end{tikzpicture}
    \caption{A Markov chain for $N=3$ and $r=2$.}
    \label{chain}
  \end{center}
\end{figure}

\begin{remark}
  Given two states of the 2-PASEP $x$ and $y$, using the different
  transitions one can see that the transition from $x$ to $y$ is the
  same as the one from $\invol(x)$ to $\invol(y)$ where $\invol$ is
  the map that reverse a state and replace the $\circ$ spots by
  $\bullet$ and conversely. This is called the particle hole symmetry
  of the process.
\end{remark}
For example, $P_{\bullet\graybullet\graybullet,
  \graybullet\bullet\graybullet}$ is equal to
$P_{\graybullet\graybullet\circ, \graybullet\circ\graybullet}$.

To each state $x$ of the 2-PASEP with $N$ sites we associate a word $X(x)$
in $\{A,D,E\}^N$ using the following map:
$$ \circ\mapsto E;~ \bullet\mapsto D;~ \graybullet\mapsto A. $$

We define an involution $\invol$ on words in $\{A,D,E\}^N$ that
corresponds to the particle hole symmetry.

\begin{definition}
  \label{def:invol}
  Let $X\in\{A,D,E\}^N$, define $\invol(X)$ as the word obtained after
  reversing $X$ and replacing $D$ by $E$ and conversely.
\end{definition}

Uchiyama~\cite{uchiyama} proved that we can use a Matrix Ansatz in
order to compute the stationary distribution of the states of the
2-PASEP. We denote by $\Prob(x)$ the stationary distribution of a
state $x$.

\begin{proposition}\cite{uchiyama}
  Let $D,A,E$ be infinite matrices. Let $\bra{W}$  (resp. $\ket{V}$)
  be an infinite row (resp. column) 
  vector satisfying the Ansatz:
  \begin{eqnarray}
    DE&=&qED+D+E; \label{ansatzDE}\\ 
    DA&=&qAD+A; \label{ansatzDA}\\
    AE&=&qEA+A; \label{ansatzAE}\\
    \bra{W}E&=&\bra{W} \label{ansatzAlpha}; \\
    D\ket{V}&=&\ket{V} \label{ansatzBeta}.
  \end{eqnarray}
  Then the probability to be in a state $x$  in $\{\circ, \bullet, \graybullet\}^N$ with $r$ letters $A$ is:
  \begin{equation}
    \Prob(x) = \frac{\bra{W}X(x)\ket{V}}{[y^r]\bra{W}(D+yA+E)^N\ket{V}}.
    \label{stationProbaMat}
  \end{equation}
  where $[y^r]$ means that we consider the coefficient of the monomial $y^r$.
  \label{uchi}
\end{proposition}

\begin{remark}
  The Ansatz in~\cite{uchiyama} is more general, corresponding to the
  2-ASEP where particles may enter and exit the chain from the right
  and the left with different rates.
\end{remark}

We give here a new solution of this system using the following
matrices.

\begin{eqnarray}
  \label{matrixD}
  D&=&\begin{pmatrix}
    1  & 1 & 0 & 0&\dots \\
    0  & [2]_q & [2]_q & 0& \dots\\ 
    0  & 0 & [3]_q & [3]_q &  \dots\\ 
    0  & 0 &  0 & [4]_q & \dots\\ 
    \vdots & \vdots & \vdots & \vdots & 
  \end{pmatrix}; \\
  \label{matrixE}
  E&=&\begin{pmatrix}
    0  & 0     &  0    & 0    &   \dots\\
    1  & 1     & 0     & 0    & \dots \\
    0  & [2]_q & [2]_q & 0     & \dots\\ 
    0  & 0     & [3]_q & [3]_q &  \dots\\  
    \vdots & \vdots & \vdots & \vdots & 
  \end{pmatrix}; \\
  \label{matrixA}
  A&=&\begin{pmatrix}
  1  & 0 & 0   & 0   &   \ldots\\
  0  & q & 0   & 0   & \ldots \\
  0  & 0 & q^2 & 0   & \ldots\\ 
  0  & 0 & 0   & q^3 &  \ldots\\  
  \vdots & \vdots & \vdots & \vdots & 
  \end{pmatrix}(D+E);\\
  \label{vectorW}
  \bra{W}&=&(1,1,0,0,\dots); \\
  \label{vectorV}
  \ket{V}&=&\begin{pmatrix}
    1 \\
    0 \\
    0 \\
    \vdots  
  \end{pmatrix}.
\end{eqnarray}

\begin{lem}
  The previous matrices and vectors satisfy the equations from
  Proposition~\ref{uchi}.
\end{lem}

\begin{proof}
  Let $M=DE$. For all $i$, $j$, we have
  \begin{eqnarray*}
    M_{ij}&=& \sum_{k=1}^\infty D_{ik}E_{kj} \\
         &=& \sum_{k=i}^{i+1}[i]_qE_{kj}.
  \end{eqnarray*}
  Hence, $M$ is described as follow :
  \begin{equation}
    \label{DE}
    \left\{\begin{array}{rcccl}
    M_{i,i-1} &=& [i]_qE_{i,i-1} &=& [i]_q[i\!-\!1]_q;\\
    M_{i,i}   &=& [i]_q\Big(E_{i,i}+E_{i+1,i}\Big) &=&
    [i]_q\Big([i\!-\!1]_q+[i]_q\Big);\\
    M_{i,i+1} &=& [i]_qE_{i+1,i+1} &=& [i]_q^2;
    \end{array}\right.
  \end{equation}
  and $M_{i,j}=0$ if $|i-j|>1$.
  Let $N=ED$. We have
  \begin{equation*}
    \left\{\begin{array}{rcl}
    N_{i,i-1} &=& [i\!-\!1]_q^2;\\[0.15cm]
    N_{i,i}   &=& [i\!-\!1]_q\Big([i\!-\!1]_q+[i]_q\Big);\\[0.2cm]
    N_{i,i+1} &=& [i\!-\!1]_q[i]_q;
    \end{array}\right.
  \end{equation*}
  and $N_{i,j}=0$ if $|i-j|>1$.
  For all $i$, we have
  \begin{equation*}
    \begin{array}{rcccl}
      qN_{i,i-1}+D_{i,i-1}+E_{i,i-1}&=&[i\!-\!1]_q(q[i\!-\!1]_q+1)&=&[i\!-\!1]_q[i]_q;\\[0.1cm]
      qN_{i,i}+D_{i,i}+E_{i,i}&=&(q[i\!-\!1]_q+1)([i]_q+[i\!-\!1]_q)&=&
      [i]_q([i\!-\!1]_q+[i]_q);\\[0.1cm]
      qN_{i,i+1}+D_{i,i+1}+E_{i,i+1}&=&(q[i\!-\!1]_q+1)[i]_q&=&[i]_q^2,
    \end{array}
  \end{equation*}
  which is equal to~\eqref{DE} and so~\eqref{ansatzDE} is satisfied.

  For~\eqref{ansatzDA} and~\eqref{ansatzAE}, recall that $A$ is
  described by
  \begin{equation}
    \left\{\begin{array}{rcl}
    A_{i,i-1} &=& q^i[i\!-\!1]_q;\\[0.15cm]
    A_{i,i}   &=& q^i\Big([i\!-\!1]_q+[i]_q\Big);\\[0.2cm]
    A_{i,i+1} &=& q^i[i]_q.
    \end{array}\right.
  \end{equation}
  Hence, we have
  \begin{eqnarray*}
    (DA)_{ij}&=& \sum_{k=1}^\infty D_{ik}A_{kj} \\
         &=& \sum_{k=i}^{i+1}[i]_qA_{kj}.
  \end{eqnarray*}
  We use the following description:
  \begin{equation}
    \label{DA}
    \left\{\begin{array}{rcccl}
    (DA)_{i,i-1} &=& [i]_qA_{i,i-1} &=& q^i[i]_q[i\!-\!1]_q;\\[0.15cm]
    (DA)_{i,i}   &=& [i]_q\Big(A_{i,i}+A_{i+1,i}\Big) &=&
    [i]_q\Big(q^i[i\!-\!1]_q+(q^i+q^{i+1})[i]_q\Big);\\[0.2cm]
    (DA)_{i,i+1} &=& [i]_q\Big(A_{i,i+1}+A_{i+1,i+1}\Big) &=&
    [i]_q\Big((q^i+q^{i+1})[i]_q + q^{i+1}[i\!+\!1]_q\Big); \\[0.2cm]
    (DA)_{i,i+2} &=& [i]_qA_{i+1,i+2} &=& q^{i+1}[i]_q[i\!+\!1]_q.
    \end{array}\right.
  \end{equation}
  The product $AD$ is described using the same idea:
  \begin{equation*}
    \left\{\begin{array}{rcccl}
    (AD)_{i,i-1} &=& A_{i,i-1}D_{i-1,i-1} &=& q^i[i\!-\!1]_q^2;\\[0.15cm]
    (AD)_{i,i}   &=& A_{i,i-1}D_{i-1,i}+A_{i,i}D_{i,i} &=&
    q^i\Big([i\!-\!1]_q^2 + [i\!-\!1]_q[i]_q+[i]_q^2\Big);\\[0.2cm]
    (AD)_{i,i+1} &=& A_{i,i}D_{i,i+1}+A_{i,i+1}D_{i+1,i+1} &=&
    q^i\Big([i\!-\!1]_q[i]_q + [i]_q^2 + [i]_q[i\!+\!1]_q\Big); \\[0.2cm]
    (AD)_{i,i+2} &=& A_{i,i+1}D_{i+1,i+2} &=& q^{i}[i]_q[i\!+\!1]_q.
    \end{array}\right.
\end{equation*}
Hence, for all $i$, we have
  \begin{equation*}
    \begin{array}{rcl}
      q(AD)_{i,i-1}+A_{i,i-1}&=&q^i[i\!-\!1]_q(q[i\!-\!1]_q+1)\\[0.1cm]
      &=&q^i[i\!-\!1]_q[i]_q;\\[0.2cm]
      q(AD)_{i,i}+A_{i,i}&=&q^i(q[i\!-\!1]_q+1)([i]_q
      +[i\!-\!1]_q)+q^{i+1}[i]_q^2 \\[0.1cm]
      &=& q^i[i]_q([i\!-\!1]_q+[i]_q)+q^{i+1}[i]_q^2;\\[0.2cm]
      q(AD)_{i,i+1}+A_{i,i+1}&=&q^i(q[i\!-\!1]_q+1)[i]_q
      +q^{i+1}([i]_q^2+[i]_q[i\!+\!1]_q)\\[0.1cm]
      &=&q^i[i]_q^2+q^{i+1}([i]_q^2+[i]_q[i\!+\!1]_q)~\\[0.2cm]
      q(AD)_{i,i+2}+A_{i,i+2}&=& q^{i+1}[i]_q[i\!+\!1]_q + 0 \\[0.1cm]
      &=& q^{i+1}[i]_q[i\!+\!1]_q,
    \end{array}
  \end{equation*}
  which is equal to~\eqref{DA} and proves that~\eqref{ansatzDA} is
  satisfied.

  Let us now consider the product $AE$. We have the following
  description:
  \begin{equation}
    \label{AE}
    \left\{\begin{array}{rcccl}
    (AE)_{i,i-2} &=& A_{i,i-1}E_{i-1,i-2} &=&
    q^i[i\!-\!1]_q[i-2]_q;\\[0.15cm]
    (AE)_{i,i-1} &=& A_{i,i-1}E_{i-1,i-1}+A_{i,i}E_{i,i-1} &=&
    q^i\Big([i\!-\!1]_q[i-2]_q+[i\!-\!1]_q^2+[i]_q[i\!-\!1]_q \Big);\\[0.15cm]
    (AE)_{i,i} &=& A_{i,i}E_{i,i} + A_{i,i+1}E_{i+1,i} &=&
    q^i\Big([i\!-\!1]_q^2+[i]_q[i\!-\!1]_q+[i]_q^2\Big);\\[0.15cm]
    (AE)_{i,i+1} &=& A_{i,i+1}E_{i+1,i+1} &=&
    q^i[i]_q^2.\\[0.15cm]
    \end{array}\right.
  \end{equation}
  We also have $EA$ described as follow:
  \begin{equation*}
    \left\{\begin{array}{rcccl}
    (EA)_{i,i-2} &=& E_{i,i-1}A_{i-1,i-2} &=&
    q^{i-1}[i\!-\!1]_q[i-2]_q;\\[0.15cm]
    (EA)_{i,i-1} &=& E_{i,i-1}A_{i-1,i-1}+E_{i,i}A_{i,i-1} &=&
    q^{i-1}[i-2]_q[i\!-\!1]_q + q^{i-1}[i\!-\!1]_q^2+q^i[i\!-\!1]_q^2;\\[0.15cm]
    (EA)_{i,i} &=& E_{i,i-1}A_{i-1,i} + E_{i,i}A_{i,i} &=&
    q^{i-1}[i\!-\!1]_q^2+q^i[i\!-\!1]_q^2+q^i[i]_q[i\!-\!1]_q;\\[0.15cm]
    (EA)_{i,i+1} &=& E_{i,i}A_{i,i+1} &=&
    q^i[i\!-\!1]_q[i]_q.\\[0.15cm]
    \end{array}\right.
\end{equation*}
Hence, for all $i$, we have
  \begin{equation*}
    \begin{array}{rcl}
      q(EA)_{i,i-2}+A_{i,i-2}&=&q^i[i\!-\!1]_q[i-2]_q + 0\\[0.1cm]
      &=&q^i[i\!-\!1]_q[i-2]_q;\\[0.2cm]
      q(EA)_{i,i-1}+A_{i,i-1}&=&q^i[i\!-\!1]_q[i-2]_q + q^i[i\!-\!1]_q^2
      +q^i(q[i\!-\!1]_q+1)[i\!-\!1]_q \\[0.1cm]
      &=& q^i[i\!-\!1]_q[i-2]_q + q^i[i\!-\!1]_q^2 + q^i[i]_q[i\!-\!1]_q;\\[0.2cm]
      q(EA)_{i,i+1}+A_{i,i+1}&=& q^i[i\!-\!1]_q^2
      +q^{i}(q[i\!-\!1]_q+1)([i\!-\!1]_q+[i]_q)\\[0.1cm]
      &=&q^i[i\!-\!1]_q^2+q^i[i]_q([i\!-\!1]_q+[i]_q)~\\[0.2cm]
      q(EA)_{i,i+2}+A_{i,i+2}&=& q^i(q[i\!-\!1]_q+1)[i]_q \\[0.1cm]
      &=& q^{i}[i]_q^2,
    \end{array}
  \end{equation*}
  which is equal to~\eqref{AE} and proves that~\eqref{ansatzAE} is
  satisfied. 

  Using the structure of the vectors, one easily proves
  that~\eqref{ansatzAlpha} and~\eqref{ansatzBeta}  are satisfied,
  which ends the proof.
\end{proof}

\begin{remark}
  When $q=1$, $A=D+E$ satisfies \eqref{ansatzDA} and \eqref{ansatzAE}.
  In this case we can compute the stationary distribution of the 2-PASEP
  using the stationary distribution of the PASEP (case with zero
  $\graybullet$ particles).
  Let~$x$ be a state of the $2$-PASEP with~$N$ sites and~$r$ gray
  particles. We have
  \begin{equation}
    \label{proba2PASEP}
    \Prob(x) = \frac{1}{\binom{N}{r}}\sum_y\Prob(y),
  \end{equation}
  where the sum is on all the states of the PASEP with
  particle~$\bullet$ at
  the positions of the~$\bullet$ particles of~$x$, empty sites $\circ$
  at the
  positions of the empty sites $\circ$ of~$x$ and a particle $\bullet$
  or or an empty site $\circ$ at the
  sites with a~$\graybullet$ particle in~$x$.

  For example, we consider the state~$x=\bullet\graybullet\circ$.
  Let~$x_1=\bullet\bullet\circ$ and~$x_2=\bullet\circ\circ$
  be the states associated. We have
  $\Prob(x) = \frac{14}{3\cdot 4!}.$
  Similarly, we have $\Prob(x_1)=\Prob(x_2)=\frac{7}{4!}$, and
  $$
  \frac{14}{3\cdot
    4!}=\frac{1}{\binom{3}{1}}\left(\frac{7}{4!}+\frac{7}{4!}\right).
  $$
\end{remark}

\section{Path interpretation}
\label{sec:paths}

One way to obtain a combinatorial interpretation of the stationary
distribution is to interpret each monomial of the numerator and
denominator of~\eqref{stationProbaMat} as a weighted path. 

We start by defining a new class of paths generalizing the Laguerre
histories.

\begin{definition}
  A {\em marked Laguerre history} of size $(n,r)$ is a Laguerre
  history of size~$n$ where all the steps but the first can be marked
  and~$r$ steps are marked. Any marked step starting from height~$h$
  increases its weight by~$h$.

  This way, a step $\longrightarrow$ or $\nearrow$ has a weight
  between~$0$ and~$h$ if it is not marked and between~$h$ and~$2h$
  otherwise. A step $\dashrightarrow$ or $\searrow$ has a weight
  between~$0$ and~$h-1$ if it is not marked and between~$h$ and~$2h-1$
  otherwise.
\end{definition}

To avoid confusion, all notations for the marked Laguerre histories are represented by
bold letters. An example of a marked Laguerre history of size $(8,2)$ is given
in Figure~\ref{exChemin}. The steps with overlined weight are the marked
steps.

\begin{figure}[h!t]\centering
  \begin{tikzpicture}
  \draw [thin, gray, dotted] (0, 0) grid (8, 3);
  \draw [thick] (0, 0) -- (1, 1) node[midway,scale=1,above]{};
  \draw [thick] (1, 1) -- (2, 1) node[midway,scale=1,above]{$\overline{1}$};
  \draw [thick] (2, 1) -- (3, 2) node[midway,scale=1,above]{};
  \draw [thick,dashed] (3, 2) -- (4, 2) node[midway,scale=1,above]{$\overline{3}$};
  \draw [thick] (4, 2) -- (5, 2) node[midway,scale=1,above]{};
  \draw [thick] (5, 2) -- (6, 1) node[midway,scale=1,above]{$1$};
  \draw [thick] (6, 1) -- (7, 1) node[midway,scale=1,above]{};
  \draw [thick] (7, 1) -- (8, 0) node[midway,scale=1,above]{};
\end{tikzpicture}
  \caption{An example of a Laguerre history of ${\mathfrak
      H}(ADADEDE)$ with total weight $5$.}
  \label{exChemin}
\end{figure}
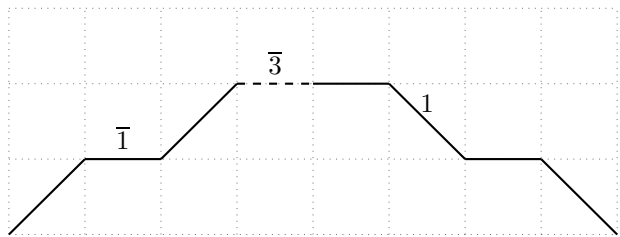

The \emph{total weight} of a marked Laguerre history $\Hg$ (denoted by
$\tw(\Hg)$) is the sum of the weights of it steps. For example, the
total weight of the marked Laguerre history in Figure~~\ref{exChemin}
is $5$. We associate a word of length~$n-1$ to a marked Laguerre
history~$\Hg$ of length~$n$ the following way. The marked steps are
labeled by~$A$, the $\nearrow$ or $\longrightarrow$ steps are
labeled~$D$ and the remaining steps are labeled~$E$. We forget the
label of the first step as it is always~$D$. We call this word the
label of~$\Hg$ and we denote it by $\etiq(\Hg)$. For example, the
label of the marked Laguerre history in Figure~\ref{exChemin} is
$ADADEDE$.

Given a word $X$, let ${\mathfrak H}(X)$ be the set of marked Laguerre
histories with label $X$ and let ${\mathcal Z}_X$ be the generating
polynomial of all the paths:
\begin{equation}
{\mathcal Z}_X(q)=\sum_{H\in {\mathfrak H}(X)} {q^{\rm wt(H)}};
\end{equation}
and let
\begin{equation}
{\mathcal Z}_{N,r}(q)=\sum_X {\mathcal Z}_{X}(q)
\end{equation}
where the sum is over all the words in $\{A,D,E\}^N$ with $r$ letters
$A$ (with $N=n-1$).

The following result gives us a combinatorial interpretation of the
steady-state probabilities of the 2-PASEP in terms of marked Laguerre
histories.

\begin{thm}
  Let $x$ be a state of the 2-PASEP with $N$ sites and $r$ gray
  particles and $X$ be the associated word in $A$, $D$, and $E$. We
  have:
  \begin{equation}
    \Prob(x) = \frac{{\mathcal Z}_{X}(q)}{{\mathcal Z}_{N,r}(q)}.
  \end{equation}
  \label{probaChemin}
\end{thm}

\begin{proof} The idea is to associate a marked Laguerre history with
each monomial of the matrix product of the numerator of
\eqref{stationProbaMat}. Any monomial corresponds to the product of
$N$ non-zero coefficients $(X_{k})_{i_k, j_k}$ where $X_k \in \{A, D,
E\}$ is the matrix corresponding to the $k$'th letter of $X$.  As the
indices $(i_k, j_k)$ must satisfy $i_k = j_{k-1}$, they can represent
the successive heights of a path:~$i_k$ corresponds to the starting
height of the $k$'th step and and~$j_k$ corresponds to its ending
height. Moreover, as the matrices~$A$, $D$, and~$E$ are tridiagonal,
$|i_k - j_{k}| \leq 1$ so the paths are Motzkin paths. In order to
have a path starting from height $0$, we need to add a $\nearrow$ or
$\longrightarrow$ step at the beginning of the path depending on which
coefficient of $\bra{W}$ has been extracted. For the steps labeled $D$, we
have $j_k \in \{i_k, i_k+1\}$ so the possible steps are $\nearrow$ or
$\longrightarrow$; for the steps labeled $E$, we have $j_k \in \{i_k,
i_k-1\}$ so the possible steps are $\searrow$ or
$\dashrightarrow$. For the steps labeled~$A$, a coefficient either
comes from the matrix~$D$ or the matrix~$E$. We choose the
corresponding step and mark it in order to be able to invert the
process.

The weight of the $k$'th step of the path corresponds to the power
of $q$ taken in the coefficient $(X_{k})_{i_k, j_k}$. One can see that
for the matrix $D$ (steps $\nearrow$ or $\longrightarrow$) the
possible weights are between $0$ and $i_k$ and that for the matrix $E$
(steps $\searrow$ or $\dashrightarrow$) the possible weights are
between $0$ and $i_k-1$.  Finally for the matrix $A$, the weights are
the same than for $D$ and $E$ on which we added $i_k$ due to the
$\mathbf{diag}(1,q,q^2,q^3,\ldots )$ factor. This proves that the
paths we obtain are exactly the marked Laguerre histories.
\end{proof}

Special cases of Theorem~\ref{probaChemin} are:
\begin{eqnarray}
  {\mathcal Z}_{N,r}(1) & = & {N\choose r}(N+1)! \\
  {\mathcal Z}_{r,r}(q) & = & [r+1]_q!
  \label{heine}
\end{eqnarray}
The first equation immediately follows from the next section which
exhibits a bijection between these marked Laguerre histories and
partially signed permutations. The second equation follows from a
continued fraction proven by Heine. A bijective proof was given by
Biane~\cite{Bi}. Another way to prove it is by using the following
lemma.

\begin{lem}
  If we denote by $m_n^k(q)$ the sum of the weights of the marked
  Laguerre histories of size $n$ with $n-1$ marked steps that end at
  height $k$ (the ending height of the last step is  $k$), we have
  \begin{equation}
    \label{eq:enumLaguerre}
    m_n^k(q) = q^{\binom{k}{2}}\frac{[n]_q!^2}{[n-k]_q![k]_q!}
  \end{equation}
\end{lem}

\begin{proof}
  We prove this lemma by induction. If $n=1$, the possible values for
  $k$ are $0$ and $1$. In both cases there is only one path of weight
  $1$ as the first step is never marked.

  Suppose the property true for $n-1$. A path of length $n$ ending at
  height $k$ can be either:
  \begin{itemize}
  \item a path of length $n-1$ ending at height $k+1$ followed by a
    $\searrow$ step;
  \item a path of length $n-1$ ending at height $k$ followed by a
    $\dashrightarrow$ step or a $\longrightarrow$ step;
  \item a path of length $n-1$ ending at height $k-1$ followed by a
    $\nearrow$ step.
  \end{itemize}
  Hence, we have
  \begin{equation}
    \label{eq:recMnk}
    m_n^k  =  m_{n-1}^{k+1}q^{k+1}[k+1]_q + m_{n-1}^{k}q^{k}\Big([k]_q
    + [k+1]_q\Big) + m_{n-1}^{k-1}q^{k-1}[k]_q.
  \end{equation}
  Using \eqref{eq:enumLaguerre} to compute $m_{n-1}^{k+1}$,
  $m_{n-1}^{k}$, and $m_{n-1}^{k-1}$ we have the following.
  \begin{equation}
    \label{eq:coeffsMnk}
    \left\{\begin{array}{rcl}
      m_{n-1}^{k+1}& = & \displaystyle
       \frac{[n-1]_q!^2}{[n-k]_q![k]_q!}q^{\binom{k}{2}+2k+1}[n-k]_q[n-k-1]_q\\[13pt]
      m_{n-1}^{k}& = & \displaystyle
                       \frac{[n-1]_q!^2}{[n-k]!_q[k]_q!}q^{\binom{k}{2}+k}[n-k]_q
                       \Big([k]_q+[k+1]_q\Big)\\[13pt]
      m_{n-1}^{k-1}& = & \displaystyle
       \frac{[n-1]_q!^2}{[n-k]_q![k]_q!}q^{\binom{k}{2}}[k]_q^2\\
    \end{array}\right.
  \end{equation}

  Using the fact that $q^k[n-k]_q = [n]_q-[k]_q$,~\eqref{eq:recMnk}
  becomes
  \begin{equation*}
    q^{\binom{k}{2}}\frac{[n-1]_q!^2}{[n-k]_q[k]_q}\left(
      ([n]_q-[k]_q)([n]_q-[k+1]_q) + ([n]_q-[k]_q)([k]_q+[k+1]_q)
      +[k]_q^2\right)
  \end{equation*}
  which simplifies to~\eqref{eq:enumLaguerre}.

  Note that the special cases $k=0$, $k=n-1$, and $k=n$ are correctly
  treated. Indeed, in~\eqref{eq:coeffsMnk} we have $m_{n-1}^{k+1}$
  equals $0$ for $k=n$ and $k=n-1$, we have $m_{n-1}^{k}$ equals $0$
  for $k=n$, and $m_{n-1}^{k-1}$ equals $0$ for $k=0$.
\end{proof}

We also give a recurrence satisfied by $\mathcal{Z}_X(q)$.
\begin{proposition}
  \label{propRec}
  Let~$X$ be a word of size~$N$ in the letters~$A$, $D$, and~$E$
 and $s$ be an
  integer. Denote by~$k$ the number of letters~$A$ or~$E$ in~$X$, we
  have
  \begin{eqnarray}
    \label{propRec_id1}
    \mathcal{Z}_{A^s\cdot D\cdot X}(q) &=& [k+1]_q\mathcal{Z}_{A^s\cdot X}(q)
    +\sum_{X=X_1\cdot E\cdot X_2}q^{\kappa(X_1)}
    \mathcal{Z}_{A^s\cdot X_1\cdot D\cdot X_2}(q);\\
    \label{propRec_id2}
    \mathcal{Z}_{A^s\cdot E\cdot X}(q) &=& [s+1]_q\mathcal{Z}_{A^s\cdot
      X}(q);\\
    \label{propRec_id3}
    \mathcal{Z}_{A^s}(q) &=& [s+1]_q!,
  \end{eqnarray}
  where $A^s$ is the word with~$s$ times the letter~$A$
  and~$\kappa(X_1)$ is the number of~$E$ and~$A$ in~$X_1$. 
\end{proposition}

\begin{proof}
  Note that $\mathcal{Z}_{A^s}(q)=\mathcal{Z}_{r,r}(q)$, such
  that~\eqref{propRec_id3} directly comes from~\eqref{heine}.

  To prove the other parts of this property, we use the fact that
  $\mathcal{Z}_Y(q)$ is equal to the matrix product
  $\bra{W}m(Y)\ket{V}$ where~$m$ is the morphism sending the letters
  of~$Y$ to the matrix satisfying the Ansatz of
  Proposition~\ref{uchi}. We shall identify~$m(Y)$ and~$Y$ in the rest
  of this proof.

  We prove Equation~\eqref{propRec_id2} by induction on~$s$. If~$s=0$,
  this equation is just another way of writing~\eqref{ansatzAlpha}:
  $\bra{W}E=\bra{W}$. Otherwise, using~\eqref{ansatzAE}, $AE=qEA+A$,
  we have
  \begin{eqnarray*}
    \bra{W}A^{s}EX\ket{V} &=& q\bra{W}A^{s-1}EAX\ket{V} +
    \bra{W}A^{s-1}X\ket{V} \\ 
    &=& (q[s]_q! + 1)\bra{W}A^sX\ket{V},
  \end{eqnarray*}
  using the induction relation, so we obtain
  $$\bra{W}A^sEX\ket{V}=[s+1]_q!\bra{W}SX\ket{V}.$$

  For Equation~\eqref{propRec_id1}, we prove more generally that for
  any word~$Y$, we have
  \begin{equation}
    \label{preuveRec_id1}
    \bra{W}YDX\ket{V} = [k+1]_q\bra{W}YX\ket{V}
                           +\sum_{X=X_1EX_2}q^{\kappa(X_1)}
                           \bra{W}YX_1DX_2\ket{V}.
  \end{equation}
  It suffices then to set~$Y=A^s$ to obtain the result. We prove this
  equation by induction on the size of~$X$. If~$X$ is the empty word,
  it is simply another writing of~\eqref{ansatzBeta}:
  $D\ket{V}=\ket{V}$. Otherwise, let~$X'$ be the word obtained
  from~$X$ by removing the first letter. There are three
  possibilities.
  \begin{itemize}
  \item If~$X=DX'$, let $Y'=YD$. We then have
    $\bra{W}YDX\ket{V}=\bra{W}Y'DX'\ket{V}$. Let us prove that the
    induction relation applied to $\bra{W}Y'DX'\ket{V}$ gives us the
    same result that for $\bra{W}YDX\ket{V}$. As~$X$ and~$X'$ have the
    same number of letters~$A$ and~$E$, the first term of the right
    part of~\eqref{preuveRec_id1} equals
    $[k+1]_q\bra{W}Y'X'\ket{V}$. Moreover, we can write~$X$
    as~$X_1EX_2$ if and only if we can write~$X'$ as 
    $X_1'EX_2$ with $X_1=DX_1'$. Moreover, in these equalities, $X_1$
    and $X_1'$ have the same number of~$E$ and~$A$. Thus,
    \begin{equation*}
      \sum_{X'=X_1'EX_2}q^{\kappa(X_1')}
      \bra{W}Y'X_1'DX_2\ket{V} = \sum_{X=X_1EX_2}q^{\kappa(X_1)}
      \bra{W}Y'X_1'DX_2\ket{V},
    \end{equation*}
    so the induction is satisfied using the fact that
    $Y'X'=YX$ and $Y'X_1'=YX_1$.
  \item If~$X=AX'$, let $Y'=YA$. Using~\eqref{ansatzDA}
    ($DA=qAD+A$), we have
    \begin{equation}
      \label{YDAXonYADX}
      \bra{W}YDX\ket{V}=q\bra{W}Y'DX'\ket{V}+\bra{W}YX\ket{V}.
    \end{equation}
    As~$X'$ and~$X$ have the same number of letters~$E$, we can
    write~$X$ as $X_1EX_2$ if and only if we can write $X'$ as
    $X_1'EX_2$ with $X_1=AX_1'$. Moreover, in these equalities, $X_1$
    and $X_1'$ have the same number of~$E$ and~$X_1$ has one more~$A$
    than~$X_1'$. Hence, the induction relation implies
    \begin{eqnarray*}       
      \bra{W}Y'DX'\ket{V}  &=& [k]_q\bra{W}Y'X'\ket{V}
      +\sum_{X'=X_1'EX_2}q^{\kappa(X_1')}
      \bra{W}Y'X_1'DX_2\ket{V} \\
      &=& [k]_q\bra{W}YX\ket{V}
      +\sum_{X=X_1EX_2}q^{\kappa(X_1)-1}
      \bra{W}YX_1DX_2\ket{V},
    \end{eqnarray*}
    as $Y'X'=YX$ and $Y'X_1'=YX_1$. Using this equality
    in~\eqref{YDAXonYADX}, the induction relation is satisfied.
  \item If~$X = EX'$, let $Y'= YE$. Using~\eqref{ansatzDE}
    ($DE=qED+E+D$), we have
    \begin{equation}
      \label{YDEXonYEDX}
      \bra{W}YDX\ket{V}=q\bra{W}Y'DX'\ket{V}+\bra{W}YX\ket{V}
      +\bra{W}YDX'\ket{V}.
    \end{equation}
    As~$X$ starts with an~$E$, except for the case $X=EX'$,
    we can write $X$ as $X_1EX_2$ if and only if $X'$ can be written
    as $X_1'EX_2$ with $X_1=AX_1'$. Moreover, in these equalities,
    $X_1$ and $X_1'$ have the same number of~$A$ and $X_1$ has one
    more~$E$ than~$X_1'$. Hence, the induction relation implies:
    \begin{eqnarray*}       
      \bra{W}Y'DX'\ket{V}  &=& [k]_q\bra{W}Y'X'\ket{V}
      +\sum_{X'=X_1'EX_2}q^{\kappa(X_1')}
      \bra{W}Y'X_1'DX_2\ket{V},
    \end{eqnarray*}
    and the sum can be rewritten as
    \begin{equation*}
      \sum_{X=X_1EX_2}q^{\kappa(X_1)-1}
      \bra{W}YX_1DX_2\ket{V}-\frac{1}{q}\bra{W}YDX'\ket{V}.
    \end{equation*}
    Using this equality in~\eqref{YDEXonYEDX}, the induction relation
    is satisfied.
  \end{itemize}
  Hence, the induction relation is satisfied in every cases, which
  ends the proof.
\end{proof}

Note that this proof implies that equations~\eqref{propRec_id1}
and~\eqref{propRec_id2} are satisfied for any combinatorial object
interpreting the probabilities of the $2$-PASEP, whereas
Equation~\eqref{propRec_id3} depends on the objects themselves.

As a corollary we obtain a factorization of $Z_{X}(q)$.

\begin{corollary}
  \label{cor:Zfactor}
  Let $X$ be a word of $\{A, D,E\}^N$ with $r$ letters $A$. we have
  \begin{equation}
    Z_X(q) = [r+1]_q!\widetilde{Z_X(q)},
  \end{equation}
  where $\widetilde{Z_X(q)}$ is a polynomial in $q$ with nonnegative
  integer coefficients.
\end{corollary}

It would be interesting to find a combinatorial proof of this result
as it would allow us to give a combinatorial interpretation of the
probabilities of the 2-PASEP with a general value for $q$ using a
smaller family of objects.

In Definition~\ref{equiv} we define an equivalence relation on
partially signed permutations implying a combinatorial proof of
Corollary~\ref{cor:Zfactor} for $q=1$. 

\section{Combinatorial interpretation using partially signed
  permutations}
\label{sec:perm}

In order to obtain a combinatorial interpretation of these
probabilities in terms of generalized permutations, we shall use a
generalization of the Françon-Viennot bijection. The original
bijection, defined in~\cite{FV}, is a bijection between Laguerre
histories and permutations that we extend to a bijection between
marked Laguerre histories and partially signed permutations.

\begin{algo}
  \label{algo:FVGen}
  \begin{itemize}[before = \leavevmode, itemsep=0pt]
  \item Input: A partially signed permutation~$\ssig$ of size~$n$. 
  \item Output: A marked Laguerre history $\Hg$ of size~$n$.
  \item Execution: Let $H=\psi_{FV}(\sigma)$ where $\sigma$ is the
    permutation obtained from $\ssig$ by removing the signs. For every
    $\overline{i} \in \ssig$, mark the $i\!$'th step of $H$ to
    build~$\Hg$. When marking a step starting at height~$h$, add~$h$
    to its weight.
  \end{itemize}  
\end{algo}

We denote by $\Psi_{FV}(\ssig)$ the result of this algorithm. For
example, for $\ssig=\overline{2}57836\overline{4}1$, the image
of~$\ssig$ without the signs is given on the left of
Figure~\ref{chemin} and the marked version on the right.

\begin{figure}[h!t]\centering
  \begin{center}
   \begin{tikzpicture}
  \node (EE) at (0,0){\scalebox{0.73}{\begin{tikzpicture}
  \draw [thin, gray, dotted] (0, 0) grid (8, 3);
  \draw [thick] (0, 0) -- (1, 1) node[midway,scale=1,above]{};
  \draw [thick] (1, 1) -- (2, 1) node[midway,scale=1,above]{};
  \draw [thick] (2, 1) -- (3, 2) node[midway,scale=1,above]{};
  \draw [thick, dashed] (3, 2) -- (4, 2)node[midway,scale=1,above]{$1$};
  \draw [thick] (4, 2) -- (5, 2) node[midway,scale=1,above]{};
  \draw [thick] (5, 2) -- (6, 1) node[midway,scale=1,above]{$1$};
  \draw [thick] (6, 1) -- (7, 1) node[midway,scale=1,above]{};
  \draw [thick] (7, 1) -- (8, 0) node[midway,scale=1,above]{};
\end{tikzpicture}}};
  \node (DE) at (6.6,0){\scalebox{0.73}{\begin{tikzpicture}
  \draw [thin, gray, dotted] (0, 0) grid (8, 3);
  \draw [thick] (0, 0) -- (1, 1) node[midway,scale=1,above]{};
  \draw [thick] (1, 1) -- (2, 1) node[midway,scale=1,above]{$\overline{1}$};
  \draw [thick] (2, 1) -- (3, 2) node[midway,scale=1,above]{};
  \draw [thick, dashed] (3, 2) -- (4, 2)
  node[midway,scale=1,above]{$\overline{3}$};
  \draw [thick] (4, 2) -- (5, 2) node[midway,scale=1,above]{};
  \draw [thick] (5, 2) -- (6, 1) node[midway,scale=1,above]{$1$};
  \draw [thick] (6, 1) -- (7, 1) node[midway,scale=1,above]{};
  \draw [thick] (7, 1) -- (8, 0) node[midway,scale=1,above]{};
\end{tikzpicture}}};
\end{tikzpicture}
    \caption{Result of Françon-Viennot bijection for the
      permutation $25783641$ on the left and for the partially signed
      permutation $\ssig=\overline{2}57836\overline{4}1$ on the
      right.}
    \label{chemin}
  \end{center}
\end{figure}

The reciprocal map is obtained by storing the positions of the marked
steps in the marked Laguerre history and then applying the reciprocal
map of the usual Françon-Viennot bijection. Finally, just overline the
values corresponding to the marked steps.

\begin{proposition}
  The map $\psi_{FV}$ is a bijection between
  partially signed permutations of size $n$ with $r$ overlined values
  and marked Laguerre histories of size $n$ with $r$ marked steps.

  Moreover, let $\ssig\in B_n'$. We have
  \begin{eqnarray}
    \ade(\GC(\ssig)) &=& \etiq(\psi_{FV}(\ssig)); \\
    \tw(\ssig) &=& \tw(\psi_{FV}(\ssig)).
  \end{eqnarray}
  \label{propFV}
\end{proposition}

\begin{proof}
  As permutations and Laguerre histories are in bijection, there are
  as many partially signed permutations of size~$n$ with~$r$ overlined
  values as marked Laguerre histories of size~$n$ with~$r$ marked
  steps. Moreover, let~$\ssig$ and~$\ttau$ be two partially signed
  permutations having the same image~$\Hg$ by~$\psi_{FV}$. As the
  Françon-Viennot bijection is injective on permutations, the
  permutations~$\sigma$ and~$\tau$ obtained from~$\ssig$ and~$\ttau$
  are equal. In addition, the positions of the marked steps of~$\Hg$
  correspond to the overlined values of~$\ssig$ and~$\ttau$. They are
  therefore the same and so $\ssig=\ttau$.

  Let~$\ssig$ be a partially signed permutation and
  $\Hg=\psi_{FV}(\ssig)$, recall from Definition~\ref{def:312} that
  $\tw(\ssig)$ is the number of $\312$ patterns of $\ssig$ plus its
  number of $(31,\overline{2})$ patterns.  By construction, the number
  of $\312$ patterns of $\ssig$ is equal to
  the weight of~$\Hg$ if we remove the additional weights of the
  marked steps. In order to prove that $\tw(\ssig)=\tw(\Hg)$, we prove
  that the height~$h$ of a marked step of~$\Hg$ in position~$k$ is
  equal to the number of descents $\ssig_i >\ssig_{i+1}$ of~$\ssig$
  such that $\ssig_{i+1}<k<\ssig_i$. We use the reciprocal map of the
  Françon-Viennot bijection recalled in
  Algorithm~\ref{algoInverseFV} where~$h$ corresponds to the
  number of~$\circ$ present when we add the value~$k$ in the
  permutation. Moreover, all these positions except for the last one
  shall be occupied by a value greater than~$k$ and shall create a
  descent with the value to its right. Thus, there are~$h$
  descents satisfying the desired property.

  Finally, depending if~$k$ is overlined or not and forms a descent or
  not,  the $k$'th step of~$H$ is labeled by~$A$, $E$, or~$D$ and
  $\ade(\GC(\ssig))$ is equal to the label of $\psi_{FV}(\ssig)$.
\end{proof}

Using Theorem~\ref{probaChemin} and Proposition~\ref{propFV} we obtain
the following theorem.

\begin{thm}
  Let~$x$ be a state of the $2$-PASEP with~$N$ sites having~$r$ gray
  particles. We have
  \begin{equation}
    \Prob(x) =
    \frac{1}{Z_{N,r}(q)}\sum_{\ade(\GC(\ssig))=X(x)}q^{\tw(\ssig)},
  \end{equation}
  where $Z_{N,r}$ is the sum of $q^{\tw(\ssig)}$ for all partially
  signed permutations~$\ssig$ of size~$N+1$ having~$r$ overlined values
  \label{probaPerm}
\end{thm}

Hence, partially signed permutations can be used to describe the
probabilities of the $2$-PASEP for any value of the parameter~$q$.

In order to give a combinatorial proof of Corollary~\ref{cor:Zfactor}
in the case $q=1$, we define here an equivalence relation on partially
signed permutations.

\begin{definition}
  \label{equiv}
  Let~$\ssig$ be a partially signed permutation of size~$n$ with~$r$
  overlined values. Let~$i_1, \dots, i_{r+1}$ be the positions of
  these values and of~$1$, and let $u_1,\dots,u_{r+2}$
  be the factors of~$\ssig$ of the form
  \begin{equation}
    u_k := \ssig_{i_{k-1}+1}\dots\ssig_{i_k},
  \end{equation}
  with $i_0 = 0$ and $i_{r+2}=n$. Let~$\ttau$ be a partially signed
  permutation, we say that~$\ssig$ and~$\ttau$ are equivalent 
  ($\ssig\sim\ttau$) if  there is a permutation $\mu\in\SG_{r+1}$
  such that
  \begin{equation}
    \ttau = u_{\mu_1}\dots u_{\mu_{r+1}}u_{r+2}.
  \end{equation}
\end{definition}

For example, for $\ssig=\overline{2}73\overline{4}5186$, we have
$u_1=\overline{2}$, $u_2=73\overline{4}$, $u_3=51$ and $u_4=86$. For
$\mu=312$, we have $\ttau=51\overline{2}73\overline{4}86$. The set of
the partially signed permutations equivalent to~$\ssig$ is:
\begin{equation*}
  \{ \overline{2}5173\overline{4}86,~\overline{2}73\overline{4}5186,~
  51\overline{2}73\overline{4}86,~5173\overline{4}\overline{2}86,~
  73\overline{4}\overline{2}5186,~73\overline{4}51\overline{2}86\}.
\end{equation*}

\begin{lem}
  Let $\ssig$ and $\ttau$ be two partially signed permutations such
  that $\ssig\sim\ttau$. We have
  \begin{equation}
    \GC(\ssig)=\GC(\ttau)
  \end{equation}
\end{lem}

\begin{proof}
  Firstly, we have $\Sign(\ssig)=\Sign(\ttau)$ as~$\sim$ does not
  change the signs of the values. Moreover, the value on the right of
  a not overlined value (other than~$1$) is the same in~$\ssig$ and
  in~$\ttau$ such that $\GDes(\ssig)=\GDes(\ttau)$, which ends the
  proof.
\end{proof}

As there are $(r+1)!$ partially signed permutations in the equivalence
class of a permutation with~$r$ overlined values, this lemma implies
Corollary~\ref{cor:Zfactor} for $q=1$. Note that the above equivalence
relation leaves $\GC$ unchanged but does not modify the weight in order
to obtain the $[r+1]_q!$ factor. 

Each equivalence class may be identified by an \emph{assemblée of
  permutation}, a family of object used in~\cite{MV} to give a
combinatorial interpretation of the probabilities of the 2-PASEP in
the case $q=1$ with other parameters.

\section{Another interpretation}
\label{sectionLargeLaguerre}

In Section~\ref{secPASEP} we show that for any state $x$ of the
2-PASEP and $X=X(x)$, we have $\Prob(x)=\Prob(\invol(x))$ where
$\invol$ is the particle hole symmetry involution of
Definition~\ref{def:invol}.

Unfortunately, this property cannot be observed directly on the
marked Laguerre histories and on the partially signed permutations. In
the case of the usual PASEP, this property also exists and can be
observed using large Laguerre histories~\cite{JV}. Using the
connection between Laguerre histories and large Laguerre histories, we
define the \emph{marked large Laguerre histories} and use them to
obtain another interpretation of the probabilities.

\subsection{Marked large Laguerre histories}

Let us start by recalling the definition of large Laguerre
histories~\cite{Vi}.

\begin{definition}
  A large Laguerre history of size $n$ is a weighted Motzkin path of
  size $n$ with two different horizontal steps such that the weight of
  a step $\nearrow$, $\longrightarrow$, $\searrow$, or
  $\dashrightarrow$ starting from height $h$ is between $0$ and $h$.
\end{definition}

\begin{definition}
  A \emph{marked large Laguerre history} of size $(n,r)$ is a large
  Laguerre history of size~$n$ such that every step may be marked
  and~$r$ steps are marked. If a step $\nearrow$ or $\longrightarrow$
  starting at height~$h$ is marked, its weight is increased by~$h+1$
  where the marked steps $\searrow$ or $\dashrightarrow$ starting at
  height~$h$ increase their weight by~$h$.

  Hence, a step $\longrightarrow$ or $\nearrow$ has a weight
  between~$0$ and~$h$ if it is not marked and between~$h+1$ and~$2h+1$
  if it is marked. A step $\dashrightarrow$ or $\searrow$ has a weight
  between~$0$ and~$h$ if it is not marked and between~$h$ and~$2h$ if
  it is marked.

  The \emph{label} of a marked large Laguerre history~$\Hg$ of size
  $(n,r)$ is the word of size~$n$ in the letters~$A$, $D$ and~$E$
  obtained by sending the marked steps on~$A$, the steps~$\nearrow$
  and $\longrightarrow$ on~$D$ and the remaining steps on~$E$. We
  denote it by $\etiq(\Hg)$.
\end{definition}

Note that the marked large Laguerre histories of size $(n,0)$ are
exactly the large Laguerre histories of size $n$.

An example of a marked large Laguerre history is given in
Figure~\ref{exLargLagMark}. We denote by ${\mathfrak H}^0(X)$ the set
of all marked large Laguerre histories whose label is~$X$.

\begin{figure}[ht]
  \begin{center}
    \begin{tikzpicture}
  \draw [thin, gray, dotted] (0, 0) grid (7, 3);
  \draw [thick] (0, 0) -- (1, 1) node[midway,scale=1,above]{$\overline{1}$};
  \draw [thick] (1, 1) -- (2, 1) node[midway,scale=1,above]{};
  \draw [thick] (2, 1) -- (3, 1) node[midway,scale=1,above]{$\overline{2}$};
  \draw [thick] (3, 1) -- (4, 2) node[midway,scale=1,above]{$1$};
  \draw [thick] (4, 2) -- (5, 1) node[midway,scale=1,above]{};
  \draw [thick,dashed] (5, 1) -- (6, 1) node[midway,scale=1,above]{$1$};
  \draw [thick] (6, 1) -- (7, 0) node[midway,scale=1,above]{};
\end{tikzpicture}
    \caption{A marked large Laguerre history of size $(8,2)$ which
      label is $ADADDEDE$.}
    \label{exLargLagMark}
  \end{center}
\end{figure}
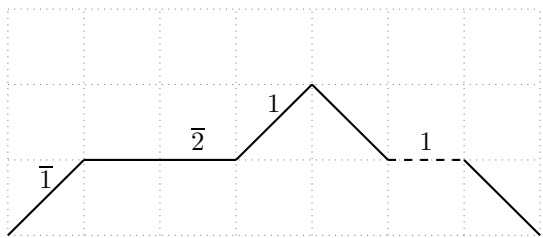

In order to use these objects to obtain a combinatorial interpretation
of the probabilities of the $2$-PASEP, we define a bijection between
marked Laguerre histories and marked large Laguerre histories. This
map is a generalization of a bijection between Laguerre histories and
large Laguerre histories described by the second author
in~\cite{Nun} using weighted Dyck path. The following algorithm is
another description of this map directly on Laguerre histories.

\begin{algo}
  \label{algo:Laguerre}
  \begin{itemize}[before = \leavevmode, itemsep=0pt]
  \item Input: a Laguerre history $H$ of size $n$.
  \item Output: a large Laguerre history~$H'$ of size $n-1$.
  \item Execution: for all $1 \leq i < n$, build $H'_i$ (the $i\!$'th
    step of $H'$) using the
    following table
    \renewcommand{\arraystretch}{2}
    \begin{equation}
      \label{eq:tableLaguerre}
      \begin{array}{c|c|c}
        H_i\backslash H_{i+1} & \nearrow \text{ or } \longrightarrow
        & \searrow \text{ or } \dashrightarrow \\ \hline
        \nearrow \text{ or } \dashrightarrow & \nearrow
        & \dashrightarrow \\ \hline
        \searrow \text{ or } \longrightarrow & \longrightarrow
        & \searrow
      \end{array}
    \end{equation}
    \renewcommand{\arraystretch}{2}
    The weight of $H'_i$ is equal to the weight of $H_i$. 
  \end{itemize}
\end{algo}

We denote by $\Psi$ the map associated with
Algorithm~\ref{algo:Laguerre}. The reciprocal map can also be
described using a similar table from large Laguerre histories of
size~$n$ to Laguerre histories of size~$n+1$. An example of $\Psi$ is
given in Figure~\ref{fig:exPsi}

\begin{figure}[ht]
  \begin{center}
    \begin{tikzpicture}
  \node (LH) at (0,0){\scalebox{0.75}{\begin{tikzpicture}
  \draw [thin, gray, dotted] (0, 0) grid (8, 3);
  \draw [thick] (0, 0) -- (1, 1) node[midway,scale=1,above]{};
  \draw [thick] (1, 1) -- (2, 1) node[midway,scale=1,above]{};
  \draw [thick] (2, 1) -- (3, 2) node[midway,scale=1,above]{};
  \draw [thick, dashed] (3, 2) -- (4, 2)node[midway,scale=1,above]{$1$};
  \draw [thick] (4, 2) -- (5, 2) node[midway,scale=1,above]{};
  \draw [thick] (5, 2) -- (6, 1) node[midway,scale=1,above]{$1$};
  \draw [thick] (6, 1) -- (7, 1) node[midway,scale=1,above]{};
  \draw [thick] (7, 1) -- (8, 0) node[midway,scale=1,above]{};
\end{tikzpicture}}};
  \node (LLH) at (7,0){\scalebox{0.75}{\begin{tikzpicture}
  \draw [thin, gray, dotted] (0, 0) grid (7, 3);
  \draw [thick] (0, 0) -- (1, 1) node[midway,scale=1,above]{};
  \draw [thick] (1, 1) -- (2, 1) node[midway,scale=1,above]{};
  \draw [thick,dashed] (2, 1) -- (3, 1) node[midway,scale=1,above]{};
  \draw [thick] (3, 1) -- (4, 2) node[midway,scale=1,above]{$1$};
  \draw [thick] (4, 2) -- (5, 1) node[midway,scale=1,above]{};
  \draw [thick,dashed] (5, 1) -- (6, 1) node[midway,scale=1,above]{$1$};
  \draw [thick] (6, 1) -- (7, 0) node[midway,scale=1,above]{};
\end{tikzpicture}}};

  \draw[->] (LH) -- (LLH) node[midway,above] {$\Psi$};

\end{tikzpicture}
    \caption{An example of $\Psi$ on a Laguerre history of size 8.}
    \label{fig:exPsi}
  \end{center}
\end{figure}
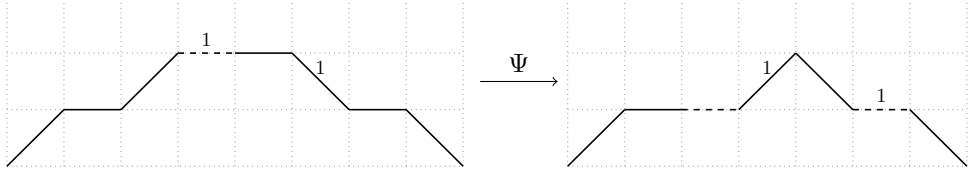

The map $\Psi$ is a variant of a well known bijection between Laguerre
histories and large Laguerre histories using their common
representation as weighted Dyck paths~\cite{Nun}. The proof that it is a
bijection uses the following lemma.

\begin{lem}
  \label{lem:psiBij}
  Let $H$ be a Laguerre history of size $n$ and $0<k<n$. Let $h_k$ be
  the height of the $k\!$'th step of $H$ and $h'_k$  be
  the height of the $k\!$'th step of $\Psi(H)$
  \begin{itemize}
  \item If $H_k$ is  $\nearrow$ or $\longrightarrow$, then $h'_k=h_k$.
  \item If $H_k$ is  $\searrow$ or $\dashrightarrow$, then $h'_k=h_k-1$.
  \end{itemize}
\end{lem}

\begin{proof}
  We prove this lemma by induction on $k$. Let $H'=\Psi(H)$. We
  have $h_1=h'_1=0$ and $H_1$ is indeed $\nearrow$ or
  $\longrightarrow$. In the case 
  $k=2$, if $H_2$ is $\nearrow$ or $\longrightarrow$ then, using the
  table in~\eqref{eq:tableLaguerre}, $H'_1=H_1$ as it cannot be $\searrow$ or
  $\dashrightarrow$. This proves that $h_2=h'_2$. Moreover, if $H_2$ is
  $\searrow$ or $\dashrightarrow$ then $H_1$ is necessarily $\nearrow$
  and so $H'_1$ is $\dashrightarrow$ and $h'_2=0 = h_2-1$.

  Suppose the lemma true for $k-1$, there are four cases obtained
  directly from~\eqref{eq:tableLaguerre}.
  \begin{itemize}
  \item If $H_{k-1}$ is  $\nearrow$ or $\longrightarrow$ and $H_k$ is
    $\nearrow$ or $\longrightarrow$ then $H'_{k-1}=H_{k-1}$ and
    $h'_{k-1}=h_{k-1}$ which implies $h'_{k}=h_{k}$.
  \item If $H_{k-1}$ is  $\searrow$ or $\dashrightarrow$ and $H_k$ is
    $\nearrow$ or $\longrightarrow$ then $H'_{k-1}=H_{k-1}$ and
    $h'_{k-1}=h_{k-1}-1$ which implies $h'_{k}=h_{k}-1$.
  \item If $H_{k-1}$ is  $\nearrow$ or $\longrightarrow$ and $H_k$ is
    $\searrow$ or $\dashrightarrow$ then $h'_{k-1}=h_{k-1}$ and
    $h'_{k}=h_{k}-1$.
  \item If $H_{k-1}$ is  $\searrow$ or $\dashrightarrow$ and $H_k$ is
    $\searrow$ or $\dashrightarrow$ then $h'_{k-1}=h_{k-1}-1$ and
    $h'_{k}=h_{k}$.
  \end{itemize}
\end{proof}

\begin{lem}
  \label{labelInvol}
  Let~$H$ be a Laguerre history. We have
  \begin{eqnarray}
    \label{statInvol}
    \etiq(H) &=& \etiq(\Psi(H));\\
    \label{weightInvol}
    \tw(H) &=& \tw(\Psi(H)).
  \end{eqnarray}
\end{lem}

\begin{proof}
  The weight stays unchanged when applying~$\Psi$ which proves
  Equation~\eqref{weightInvol}.

  Moreover, using~\eqref{eq:tableLaguerre} we have that for $i<n$, the
  $i$'th step of $\Psi(H)$ is $\longrightarrow$ or $\nearrow$ if and
  only if it is the case of the $(i+1)$-st step of $H$. This
  proves~\eqref{statInvol}.
\end{proof}

Note that this result is a reformulation of the first two points of
Proposition~3.6 of~\cite{Nun}.

Let us now extend this map to the marked versions of these objects.
\begin{algo}
  \label{LagToLargLag}
  \begin{itemize}[before = \leavevmode, itemsep=0pt]
  \item Input: a marked Laguerre history $\Hg$.
  \item Output: a marked large Laguerre history~$\Hg'$.
  \item Execution:
    \begin{itemize}
    \item let $H$ be the Laguerre history obtained from~$\Hg$
      by removing its marks;
    \item $H'=\Psi(H)$;
    \item for all~$k$ such that~$\Hg_k$ is marked, mark the
      $(k\!-\!1)\!$'th step of~$H'$ to build~$\Hg'$;
    \item for all marked steps of~$\Hg$, if it is a $\longrightarrow$
      (resp. $\dashrightarrow$), change it in $\dashrightarrow$
      (resp. $\longrightarrow$).
    \end{itemize}
  \end{itemize}
\end{algo}

As the execution of this algorithm is the same as the execution of
Algorithm~\ref{algo:Laguerre} when applied to Laguerre history, this
one is a generalization to marked Laguerre histories that we also
denote by~$\Psi$. An example of an execution of~$\Psi$ on
a marked Laguerre history is given on Figure~\ref{exPsiMark}. In this
example, the Laguerre history $H$ and the large Laguerre history $H'$
are the ones of Figure~\ref{fig:exPsi}.

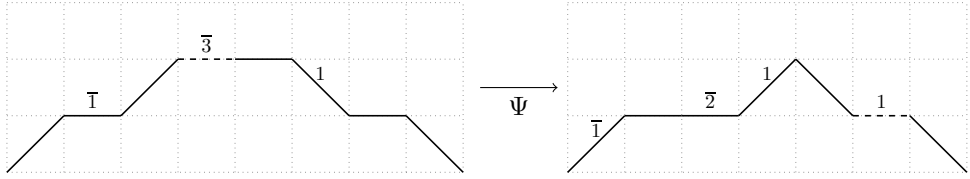
\begin{figure}[ht]
  \begin{center}
    \begin{tikzpicture}
  \node (LH) at (0,0){\scalebox{0.75}{\begin{tikzpicture}
  \draw [thin, gray, dotted] (0, 0) grid (8, 3);
  \draw [thick] (0, 0) -- (1, 1) node[midway,scale=1,above]{};
  \draw [thick] (1, 1) -- (2, 1) node[midway,scale=1,above]{$\overline{1}$};
  \draw [thick] (2, 1) -- (3, 2) node[midway,scale=1,above]{};
  \draw [thick, dashed] (3, 2) -- (4, 2)
  node[midway,scale=1,above]{$\overline{3}$};
  \draw [thick] (4, 2) -- (5, 2) node[midway,scale=1,above]{};
  \draw [thick] (5, 2) -- (6, 1) node[midway,scale=1,above]{$1$};
  \draw [thick] (6, 1) -- (7, 1) node[midway,scale=1,above]{};
  \draw [thick] (7, 1) -- (8, 0) node[midway,scale=1,above]{};
\end{tikzpicture}}};
  \node (LLH) at (7,0){\scalebox{0.75}{\begin{tikzpicture}
  \draw [thin, gray, dotted] (0, 0) grid (7, 3);
  \draw [thick] (0, 0) -- (1, 1) node[midway,scale=1,above]{$\overline{1}$};
  \draw [thick] (1, 1) -- (2, 1) node[midway,scale=1,above]{};
  \draw [thick] (2, 1) -- (3, 1) node[midway,scale=1,above]{$\overline{2}$};
  \draw [thick] (3, 1) -- (4, 2) node[midway,scale=1,above]{$1$};
  \draw [thick] (4, 2) -- (5, 1) node[midway,scale=1,above]{};
  \draw [thick,dashed] (5, 1) -- (6, 1) node[midway,scale=1,above]{$1$};
  \draw [thick] (6, 1) -- (7, 0) node[midway,scale=1,above]{};
\end{tikzpicture}}};

  \draw[->] (LH) -- (LLH) node[midway,below] {$\Psi$};
\end{tikzpicture}
    \caption{A marked Laguerre history on the left and its image
      by~$\Psi$ on the right. The intermediate steps correspond to the
      paths in Figure~\ref{fig:exPsi}.}
    \label{exPsiMark}
  \end{center}
\end{figure}

\begin{proposition}
  The map~$\Psi$ sending marked Laguerre histories to marked large
  Laguerre histories is a bijection. Moreover, for every marked Laguerre
  history~$\Hg$, we have
  \begin{eqnarray}
    \label{labelPsi}
    \etiq(\Hg) & = &\etiq(\Psi(\Hg)); \\
    \label{weightPsi}
    \tw(\Hg) & = &\tw(\Psi(\Hg)).
  \end{eqnarray}
\end{proposition}

\begin{proof}
  The fact that~$\Psi$ is a bijection comes directly from the fact
  that $\Psi$ is a bijection from Laguerre histories to large Laguerre
  histories.

  As the positions of the letters $A$ are the same in $\etiq(\Hg)$ and
  in $\etiq(\Psi(\Hg))$ and that the non-marked steps of $\psi(\Hg)$
  are the same as the steps of $\psi(H)$,
  Equation~\eqref{labelPsi} is a consequence of
  Lemma~\ref{labelInvol}.

  Let $\Hg'=\Psi(\Hg)$. Equation~\eqref{weightInvol} implies that we
  only need to prove that the weight we add to a marked step of~$\Hg'$
  corresponds to the added weight as we mark the corresponding step
  in~$\Hg$ in order to prove~\eqref{weightPsi}. 

  Let $(h_1,\dots, h_n)$ and $(h'_1,\dots,h'_{n-1})$ be such that $h_i$
  (resp. $h'_i$) is the height of the $i$'th step of $\Hg$
  (resp. $\Hg'$). Let~$k$ be the position of a marked step
  of~$\Hg$. Let~$H$ be the Laguerre history obtained from~$\Hg$ by  
  removing the marked steps and their additional weight and $H'$ the
  one obtained from $\Hg'$. As we exchange the steps $\longrightarrow$
  and $\dashrightarrow$ when we mark a step and as the
  $\longrightarrow$ and $\nearrow$ have a greater weight increase when
  marked on large Laguerre histories, we need to prove
  that~$h'_{k-1}=h_k-1$ if~$H'_{k-1}$ is a 
  $\nearrow$ step or a $\dashrightarrow$ step and that~$h'_{k-1}=h_k$
  otherwise. Let us treat the four different cases using
  Lemma~\ref{lem:psiBij} and~\eqref{eq:tableLaguerre}.
  \begin{itemize}
  \item If $H'_{k-1}$ is $\nearrow$ we have
    $h'_{k-1}=h'_k-1$. Moreover, $H_k$ is $\nearrow$ step or
    $\longrightarrow$ so $h'_{k}=h_k$ and then $h'_{k-1}=h_k-1$.
  \item If $H'_{k-1}$ is $\dashrightarrow$ we have
    $h'_{k-1}=h'_k$. Moreover, $H_k$ is $\searrow$ or
    $\dashrightarrow$ so $h'_{k}=h_k-1$ and then $h'_{k-1}=h_k-1$.
  \item If $H'_{k-1}$ is $\longrightarrow$ we have
    $h'_{k-1}=h'_k$. Moreover, $H_k$ is $\nearrow$ step or
    $\longrightarrow$ so $h'_{k}=h_k$ and then $h'_{k-1}=h_k$.
  \item If $H'_{k-1}$ is $\searrow$ we have
    $h'_{k-1}=h'_k+1$. Moreover, $H_k$ is $\searrow$ or
    $\dashrightarrow$ so $h'_{k}=h_k-1$ and then $h'_{k-1}=h_k$.
    \qedhere
  \end{itemize}
\end{proof}

We then deduce the following combinatorial interpretation.
\begin{corollary}
  Let $x$ be a state of the $2$-PASEP with~$N$ sites and~$r$ gray
  particles. We have
  \begin{equation}
    \Prob(x) = \frac{1}{Z_{N,r}(q)}
    \sum_{\Hg\in{\mathfrak H}^0(X(x))}q^{\tw(\Hg)},
  \end{equation}
  where $Z_{N,r}(q)$ is the generating series of the weights of the
  marked large Laguerre histories of size~$N$ with~$r$ marked steps.
\end{corollary}

\subsection{An involution on Laguerre histories}

Our main interest in this section is to describe an involution on
marked (large) Laguerre histories that behaves the same way as
$\invol$ does on the states of the 2-PASEP where $\invol$ is the map
defined in Definition~\ref{def:invol}. This goal is easier to achieve
on marked large Laguerre histories.

To describe this involution we need the notion of \emph{opposing
  steps} of a path: given $H_i=\nearrow$, the opposing step is the
$H_j=\searrow$ such that
$$j=\min_{k>i}\{H_k=\searrow~|~h_k=h_i+1\}.$$
We represented all opposing steps of a path in
Figure~\ref{fig:opposing}. For example, the opposing step of $H_5$ is
$H_9$.

\begin{figure}[ht]
  \centering
  \begin{tikzpicture}

  \draw [thick] (0, 0) -- (1, 1)  node (a0) [midway]{};
  \draw [thick] (1, 1) -- (2, 2) node (a1) [midway]{};
  \draw [thick] (2, 2) -- (3, 1) node (b1) [midway]{};
  \draw [thick, dashed] (3, 1) -- (4, 1) node[midway]{};
  \draw [thick] (4, 1) -- (5, 2) node (a2) [midway]{};
  \draw [thick] (5, 2) -- (6, 3) node (a3) [midway]{};
  \draw [thick] (6, 3) -- (7, 2) node (b3) [midway]{};
  \draw [thick] (7, 2) -- (8, 2) node [midway]{};
  \draw [thick] (8, 2) -- (9, 1) node (b2) [midway]{};
  \draw [thick] (9, 1) -- (10, 0) node (b0) [midway]{};
  
  \draw[dotted] (a0) -- (b0);
  \draw[dotted] (a1) -- (b1);
  \draw[dotted] (a2) -- (b2);
  \draw[dotted] (a3) -- (b3);
\end{tikzpicture}
  \caption{All opposing steps of a path.}
  \label{fig:opposing}
\end{figure}

\begin{algo}
  \label{algo:involLagLarge}
  \begin{itemize}[before = \leavevmode, itemsep=0pt]
  \item Input: a marked large Laguerre history $\Hg$.
  \item Output: a marked large Laguerre history $\widetilde{\Hg}$.
  \item Execution: Let $i_1, \dots, i_r$ be the positions of the
    marked steps of $\Hg$. Apply the following steps:
    \begin{enumerate}
    \item unmark the path;
    \item reverse the path $H$;
    \item exchange the weights of the opposing $\nearrow$ and
      $\searrow$;
    \item mark the path at the positions $n-i_1, \dots, n-i_r$;
    \item change the unmarked $\longrightarrow$ to $\dashrightarrow$
      and reciprocally.
    \end{enumerate}
  \end{itemize}
\end{algo}

We denote by $\invol$ the map associated with
Algorithm~\ref{algo:involLagLarge}. A detailed example of this
transformation is given in Figure~\ref{exinvolLagLarge} with the
different intermediate steps representing which step is applied
between two paths.

\begin{figure}[ht]
  \begin{center}
    \begin{tikzpicture}
  \node (Hg) at (0,0) {\scalebox{0.78}{
      \begin{tikzpicture}
  \draw [thin, gray, dotted] (0, 0) grid (7, 3);
  \draw [thick] (0, 0) -- (1, 1) node[midway,scale=1,above]{$\overline{1}$};
  \draw [thick] (1, 1) -- (2, 1) node[midway,scale=1,above]{};
  \draw [thick] (2, 1) -- (3, 1) node[midway,scale=1,above]{$\overline{2}$};
  \draw [thick] (3, 1) -- (4, 2) node[midway,scale=1,above]{$1$};
  \draw [thick] (4, 2) -- (5, 1) node[midway,scale=1,above]{};
  \draw [thick,dashed] (5, 1) -- (6, 1) node[midway,scale=1,above]{$1$};
  \draw [thick] (6, 1) -- (7, 0) node[midway,scale=1,above]{};
\end{tikzpicture}}
  };
  \node (H) at (6.5,0) {\scalebox{0.78}{
      \begin{tikzpicture}
  \draw [thin, gray, dotted] (0, 0) grid (7, 3);
  \draw [thick] (0, 0) -- (1, 1) node[midway,scale=1,above]{};
  \draw [thick] (1, 1) -- (2, 1) node[midway,scale=1,above]{};
  \draw [thick,dashed] (2, 1) -- (3, 1) node[midway,scale=1,above]{};
  \draw [thick] (3, 1) -- (4, 2) node[midway,scale=1,above]{$1$};
  \draw [thick] (4, 2) -- (5, 1) node[midway,scale=1,above]{};
  \draw [thick,dashed] (5, 1) -- (6, 1) node[midway,scale=1,above]{$1$};
  \draw [thick] (6, 1) -- (7, 0) node[midway,scale=1,above]{};
\end{tikzpicture}}
  };
  \node (Hr) at (6.5,-3.2) {\scalebox{0.78}{
      \begin{tikzpicture}
  \draw [thin, gray, dotted] (0, 0) grid (7, 3);
  \draw [thick] (7, 0) -- (6, 1) node[midway,scale=1,above]{};
  \draw [thick] (6, 1) -- (5, 1) node[midway,scale=1,above]{};
  \draw [thick,dashed] (5, 1) -- (4, 1) node[midway,scale=1,above]{};
  \draw [thick] (4, 1) -- (3, 2) node[midway,scale=1,above]{$1$};
  \draw [thick] (3, 2) -- (2, 1) node[midway,scale=1,above]{};
  \draw [thick,dashed] (2, 1) -- (1, 1) node[midway,scale=1,above]{$1$};
  \draw [thick] (1, 1) -- (0, 0) node[midway,scale=1,above]{};
\end{tikzpicture}}
  };
  \node (Hi) at (6.5,-6.4) {\scalebox{0.78}{
      \begin{tikzpicture}
  \draw [thin, gray, dotted] (0, 0) grid (7, 3);
  \draw [thick] (7, 0) -- (6, 1) node[midway,scale=1,above]{};
  \draw [thick] (6, 1) -- (5, 1) node[midway,scale=1,above]{};
  \draw [thick,dashed] (5, 1) -- (4, 1) node[midway,scale=1,above]{};
  \draw [thick] (4, 1) -- (3, 2) node[midway,scale=1,above]{};
  \draw [thick] (3, 2) -- (2, 1) node[midway,scale=1,above]{$1$};
  \draw [thick,dashed] (2, 1) -- (1, 1) node[midway,scale=1,above]{$1$};
  \draw [thick] (1, 1) -- (0, 0) node[midway,scale=1,above]{};
\end{tikzpicture}}
  };
  \node (Ht1) at (6.5,-9.6) {\scalebox{0.78}{
      \begin{tikzpicture}
  \draw [thin, gray, dotted] (0, 0) grid (7, 3);
  \draw [thick] (0, 0) -- (1, 1) node[midway,scale=1,above]{};
  \draw [thick, dashed] (1, 1) -- (2, 1) node[midway,scale=1,above]{$1$};
  \draw [thick] (2, 1) -- (3, 2) node[midway,scale=1,above]{$1$};
  \draw [thick] (3, 2) -- (4, 1) node[midway,scale=1,above]{};
  \draw [thick] (4, 1) -- (5, 1) node[midway,scale=1,above]{$\overline{2}$};
  \draw [thick] (5, 1) -- (6, 1) node[midway,scale=1,above]{};
  \draw [thick] (6, 1) -- (7, 0) node[midway,scale=1,above]{$\overline{1}$};
\end{tikzpicture}}
  };
  \node (Ht2) at (0,-9.6) {\scalebox{0.78}{
      \begin{tikzpicture}
  \draw [thin, gray, dotted] (0, 0) grid (7, 3);
  \draw [thick] (0, 0) -- (1, 1) node[midway,scale=1,above]{};
  \draw [thick] (1, 1) -- (2, 1) node[midway,scale=1,above]{$1$};
  \draw [thick] (2, 1) -- (3, 2) node[midway,scale=1,above]{$1$};
  \draw [thick] (3, 2) -- (4, 1) node[midway,scale=1,above]{};
  \draw [thick] (4, 1) -- (5, 1) node[midway,scale=1,above]{$\overline{2}$};
  \draw [thick, dashed] (5, 1) -- (6, 1) node[midway,scale=1,above]{};
  \draw [thick] (6, 1) -- (7, 0) node[midway,scale=1,above]{$\overline{1}$};
\end{tikzpicture}}
  };

  \draw[->] (Hg) -- (H) node[midway,scale=1,above]{$(1)$};
  \draw[->] (H) -- (Hr) node[midway,scale=1,right]{$(2)$};
  \draw[->] (Hr) -- (Hi)node[midway,scale=1,right]{$(3)$};
  \draw[->] (Hi) -- (Ht1) node[midway,scale=1,right]{$(4)$};
  \draw[->] (Ht1) -- (Ht2) node[midway,scale=1,above]{$(5)$};
  \draw[->] (Hg) -- (Ht2) node [midway, left] {$\invol$};
\end{tikzpicture}
    \caption{An execution of $\invol$ on
      $\Hg\in\mathfrak{H}^0(ADADEEE)$ to its image
      $\widetilde{\Hg}\in\mathfrak{H}^0(DDDEAEA)$.}
    \label{exinvolLagLarge}
  \end{center}
\end{figure}

\begin{thm}
  The map $\invol$ is an involution on marked large Laguerre
  histories. Moreover, if $\Hg$ is a marked large Laguerre history, we
  have 
  \begin{eqnarray}
    \label{involLabel}
    \invol\Big(\etiq(\Hg)\Big) & = & \etiq(\invol(\Hg));\\
    \label{involWeight}
    \tw(\Hg) & = & \tw(\invol(\Hg)).
  \end{eqnarray}
\end{thm}

\begin{proof}
  Reversing a Laguerre history exchanges the starting height of a step
  with its ending height. For horizontal steps this doesn't change
  anything. For increasing and decreasing steps, exchanging the weight
  with the opposing step ensure that the weight of every step is
  between $0$ and its starting height after reversing the path. Hence,
  $\invol$ is an involution.

  Reversing the path changes a $\searrow$ step to a $\nearrow$ one and
  conversely. Exchanging $\longrightarrow$ by
  $\dashrightarrow$ and conversely corresponds to exchanging
  $D$ and $E$ in the definition of $\invol$ on words in $\{A, D,
  E\}$. Moreover, the positions of the letter $A$ is just reversed
  which proves~\eqref{involLabel}.

  The steps one, two, and four of Algorithm~\ref{algo:involLagLarge}
  do not change the weight of the path. We only need to prove that
  marking $H$ at a position $n-i_k$ increases the weight as much as it
  is decreased by unmarking $\Hg$ at position $i_k$. If the step is
  $\longrightarrow$ or $\dashrightarrow$, its starting height does not
  change through the different steps. The steps does not change
  either. The increased weight is the same. A step $\nearrow$ is
  sent to $\searrow$ by reversing the path and conversely, this
  increases or decreases the height by one which compensate the fact
  that, starting at height $h$ a step $\nearrow$ increases its weight
  by $h+1$ and a step $\searrow$ by $h$. This
  proves~\eqref{involWeight}.
\end{proof}

The maps $\Psi$ and $\invol$ also induce an involution directly on
marked Laguerre history.

\begin{corollary}
  Let $\Hg$ be a marked Laguerre history and $\LagInvol =
  \Psi^{-1}\circ\invol\circ\Psi$, we have 
  \begin{eqnarray}
    \label{involLabelbis}
    \invol\Big(\etiq(\Hg)\Big) & = & \etiq(\LagInvol(\Hg));\\
    \label{involWeightbis}
    \tw(\Hg) & = & \tw(\LagInvol(\Hg)).
  \end{eqnarray}
\end{corollary}

\section{Conclusion}
\label{conclusion}

This combinatorial work gives rise to an algebraic interpretation
developed by the second author in \cite{Nunge}.
Thanks to this combinatorial algebra setting Nunge gives an exact
enumeration formula of the stationary
distribution of any state of the 2-ASEP.
This generalizes results on noncommutative symmetric functions
\cite{HNTT,T,NTW}.

\begin{thm}\cite{Nunge}
  Let $x$ be a state of the 2-PASEP and $\Ig$ be the segmented
  composition such that $X(x)=\ade(\Ig)$. We have
  \begin{equation}
    Z_{X(x)}=
    \sum_{\Jg\preceq\Ig}\left(\frac{-1}{q}\right)^{\ell(\Ig)-\ell(\Jg)}
    q^{-\operatorname{st}'(\Ig,\Jg)} c_\Jg,
  \end{equation}
  where $c_\Jg = [s]_q^{j_1}[s-1]_q^{j_2}\dots[1]_q^{j_s}$ with
  $j_1,\dots,j_s$ are the parts of $\Jg$ and $\operatorname{st}'$ is a
  statistic on segmented compositions.
\end{thm}

The 2-PASEP has five parameters in its most general form \cite{uchiyama}
$\alpha,\beta,\gamma,\delta$
and $q$. In this work we set $\alpha=\beta=1$ and $\gamma=\delta=0$.
We could also set the probabilities equal to $1/(N+1)$ to be equal to $u/(N+1)$
and then choose for example set $q>1$ and adjust the other parameters. This is an interesting
problem suggested by the referee. 
Some of our results
could probably be generalized in the case for $\alpha$ and $\beta$
general.
The work of Josuat-Verg\`es in the case of the classical PASEP \cite{JV}
gives an interpretation
for the Laguerre histories and the permutations. We leave as an open
problem to generalize this
to marked Laguerre histories and partially signed permutations.

Inspired by the Markov chains defined on permutations (or permutation
tableaux)
that project on the PASEP \cite{CW1}, we would like to define an
analog on the
partially signed permutations or on marked Laguerre histories. We
conjecture that it is
possible to
define such a Markov chain on partially signed permutations with $r$
signs that will project to the 2-PASEP.
Ideally the  graph would be composed of $(r + 1)!$ components. Each
component would
project to the Markov chain of 2-PASEP.

A natural generalization of the 2-PASEP is due to Cantini \cite{can}.
Given $\ell$ a positive integer,
the state of the process is all the words of length $N$ on the alphabet
$\{-\ell,\ldots ,\ell\}$.
Then particles $ij$ can become $ji$ with rate $1$ if $i>j$ and with rate
$q$ otherwise.
At the left and right border a particle $i$ can become $-i$ with a
certain rate. The 2-PASEP is
the case $\ell=1$. The matrix ansatz does not hold in this general
setting but Cantini shows that
the partition function is a specialization of a Koornwinder polynomial
\cite{can}.
We leave as an open problem to generalize partially signed permutations
to this model.

\printbibliography

\end{document}